\documentclass{article}
\usepackage{eurosym}
\usepackage{amsfonts}
\usepackage{makeidx}
\usepackage{amssymb}
\usepackage{amsmath}
\usepackage{amstext}
\usepackage{graphicx}
\usepackage{latexsym}
\usepackage{amsthm}
\usepackage{amsfonts}

\makeatletter
\newcommand{\displaybump}{\hbox to \@totalleftmargin{\hfil}}
\makeatother
\newtheorem{theorem}{Theorem}

\newtheorem{corollary}[theorem]{Corollary}

\newtheorem{definition}[theorem]{Definition}
\newtheorem{example}[theorem]{Example}

\newtheorem{lemma}[theorem]{Lemma}

\newtheorem{proposition}[theorem]{Proposition}
\newtheorem{remark}[theorem]{Remark}

\newtheorem*{definition*}{Definition}
\newtheorem*{theorem*}{Theorem}
\input{tcilatex}
\begin{document}

\title{Euclidean numbers and numerosities}
\author{Vieri Benci\thanks{%
Dipartimento di Matematica, Universit\`{a} degli Studi di Pisa, Via F.
Buonarroti 1/c, 56127 Pisa, Italy, e-mail: vieri.benci@unipi.it.}, Lorenzo
Luperi Baglini\thanks{%
Dipartimento di Matematica, Universit\`{a} degli Studi di Milano, Via C.
Saldini 50, 20133 Milano, Italy, e-mail: lorenzo.luperi@unimi.it, supported
by grant P 30821-N35 of the Austrian Science Fund FWF.}}
\maketitle

\begin{abstract}
Several different versions of the theory of numerosities have been
introduced in the literature. Here, we unify these approaches in a
consistent frame through the notion of set of labels, relating numerosities
with the Kiesler field of Euclidean numbers. This approach allows to easily
introduce, by means of numerosities, ordinals and their natural operations,
as well as the Lebesgue measure as a counting measure on the reals.
\end{abstract}

\tableofcontents

\section{Introduction}

The techniques of nonstandard analysis allow to construct several different
hyperreal fields which, for many pratical purposes, are equivalent. However,
there is a unique hyperreal field which is isomorphic to the closed real
field having the cardinality of the first strongly inaccessible\footnote{$%
\kappa $ is strongly inaccessible if it is uncountable, it is not a sum of
fewer than $\kappa $ cardinals smaller than $\kappa $ and, for all $\alpha
<\kappa $, $2^{\alpha }<\kappa $.} uncountable cardinal number. Such a field
has been introduced in \cite{keisler76} and we refer to it as to the Keisler
field.

Given a ring of sets $\mathcal{R}$ (closed for cartesian product) and a non
Archimedean field $\mathbb{K}$, the numerosity is a function 
\begin{equation}
\mathfrak{num}:\mathcal{R} \rightarrow \mathbb{K}  \label{kappa}
\end{equation}%
which satisfies the following properties:

\begin{itemize}
\item \textbf{Finite sets principle}: if $A$ is a finite set, then $%
\mathfrak{num}\left( A\right) =| A|$ ($|A|$ denotes the cardinality of $A$);

\item \textbf{Euclid's principle}: if $A\subset B$, then $\mathfrak{num}%
\left( A\right) <\mathfrak{num}\left( B\right)$;

\item \textbf{Sum principle}: if $A\cap B=\varnothing $, then $\mathfrak{num}%
\left( A\cup B\right) =\mathfrak{num}\left( A\right) +\mathfrak{num}\left(
B\right)$;

\item \textbf{Product principle}: $\mathfrak{num}\left( A\times B\right) =%
\mathfrak{num}\left( A\right) \cdot \mathfrak{num}\left( B\right) .$
\end{itemize}

The notion of numerosity has been introduced in \cite{benci95b},\cite%
{BDN2003} and developed in several directions (\cite{BDN2018},\cite{BDNF1},%
\cite{BDNB1},\cite{BDNB},\cite{BF},\cite{QSU},\cite{DNFtup},\cite{FM},\cite%
{mancu})\footnote{%
See also \cite{BF2019} for a historical survay of the ideas related to the
measure of the size of infinite sets.}. Since its beginning, numerosity
theory has been strictly related to some hyperreal field, namely the field $%
\mathbb{K}$ in (\ref{kappa}) must be hyperreal.

The aim of this paper is to relate the theory of numerosity to the Keisler
field in such a way that most of the properties investigated in the previous
papers are preserved and unified in a consistent frame.

In particular, we want at least the following three properties to be
satisfied:

\begin{itemize}
\item (\textbf{Consistency with the theory of cardinal numbers}) if $%
A,B\subset \mathbb{A}$ then%
\begin{equation*}
\left\vert A\right\vert <\left\vert B\right\vert \Rightarrow \mathfrak{num}%
\left( A\right) <\mathfrak{num}\left( B\right) .
\end{equation*}%
where $\left\vert E\right\vert $ denotese the cardinality of $E$.

\item (\textbf{Consistency with the theory of ordinal numbers}) if $\mathbf{%
COrd}$ is the set of the Cantor ordinal numbers smaller than the first
inaccessible uncountable cardinal number and $\mathbf{Num}$ is the set of
numerosities, then there is a map%
\begin{equation*}
\Psi :\mathbf{COrd}\rightarrow \mathbf{Num}
\end{equation*}%
such that

\begin{enumerate}
\item $\Psi (\sigma )=\mathfrak{num}\left( \{\Psi (\tau )\mid \tau <\sigma
\}\right) $;

\item $\Psi \left( \sigma \oplus \tau \right) =\Psi \left( \sigma \right)
+\Psi \left( \tau \right) ;$

\item $\Psi \left( \sigma \otimes \tau \right) =\Psi \left( \sigma \right)
\cdot \Psi \left( \tau \right) ,$
\end{enumerate}
\end{itemize}

where $\oplus $ and $\otimes $ denote the natural operations between ordinal
numbers (see Section \ref{SPO}).

\begin{itemize}
\item (\textbf{Consistency with the Lebesgue measure}) if $E\subset \mathbb{R%
}$ is a Lebesgue measurable set, then%
\begin{equation}
m_{L}\left( E\right) =st\left( \frac{\mathfrak{num}\left( E\right) }{%
\mathfrak{num}\left( \left[ 0,1\right) \right) }\right) ,  \label{m}
\end{equation}%
where $m_{L}\left( E\right) $ denotes the Lebesgue measure of $E$ and $%
st(\xi )$ denotes the standard part of $\xi $.
\end{itemize}

To this aim we build a field which, following \cite{BF}, we will call field
of Euclidean numbers. This field is isomorphic to the Keisler field and its
construction presents an extra structure that allows to build a numerosity
theory which satisfies, among others, the above requests.

\section{The Euclidean numbers}

In this section we introduce the field of Euclidean numbers. As we are going
to show, this is a hyperreal field constructed by means of a minor
modification of the usual superstructure construction, so to implement a
development of the theory of numerosity with certain useful peculiarities
(see Remark \ref{figa}).

\subsection{Non Archimedean fields\label{naf}}

Here, we recall the basic definitions and some facts regarding non
Archimedean fields. In the following, ${\mathbb{K}}$ will denote an ordered
field. We recall that such a field contains (a copy of) the rational
numbers. Its elements will be called numbers.

\begin{definition}
Let $\mathbb{K}$ be an ordered field. Let $\xi \in \mathbb{K}$. We say that:

\begin{itemize}
\item $\xi $ is infinitesimal if, for all positive $n\in \mathbb{N}$, $|\xi
|<\frac{1}{n}$;

\item $\xi $ is finite if there exists $n\in \mathbb{N}$ such that $|\xi |<n$%
;

\item $\xi $ is infinite if, for all $n\in \mathbb{N}$, $|\xi |>n$
(equivalently, if $\xi $ is not finite).
\end{itemize}
\end{definition}

\begin{definition}
An ordered field $\mathbb{K}$ is called non Archimedean if it contains an
infinitesimal $\xi \neq 0$.
\end{definition}

Infinitesimal numbers can be used to formalize the notion of "infinitely
close":

\begin{definition}
\label{def infinite closeness} We say that two numbers $\xi ,\zeta \in {%
\mathbb{K}}$ are \textbf{infinitely close} if $\xi -\zeta $ is
infinitesimal. In this case we write $\xi \sim \zeta $.
\end{definition}

Clearly, the relation "$\sim $" of infinite closeness is an equivalence
relation.

\begin{theorem}
If $\mathbb{K\supset R}$ is an ordered field, then it is non Archimedean and
every finite number $\xi \in \mathbb{K}$ is infinitely close to a unique
real number $r\sim \xi $, called the the \textbf{standard part} of $\xi $.
\end{theorem}

The standard part can be regarded as a function:%
\begin{equation}
st:\left\{ x\in \mathbb{K}\mid x\ \text{is finite}\right\} \rightarrow 
\mathbb{R}.  \label{sh}
\end{equation}%
Moreover, with some abuse of notation, we can extend $st\ $to all $\mathbb{K}
$ by setting%
\begin{equation*}
st\left( \xi \right) =\left\{ 
\begin{array}{cc}
+\infty & \text{if\ }\xi \ \text{is a positive infinite number;} \\ 
-\infty & \text{if\ }\xi \ \text{is a negative infinite number.}%
\end{array}%
\right.
\end{equation*}

\subsection{Construction of the Euclidean numbers}

Given any set $E$ we let $\mathbb{V}(E)$ be the superstructure on $E$,
namely the family of sets which is inductively defined as follows:

\begin{eqnarray*}
\mathbb{V}_{0}(E) &=&E; \\
\mathbb{V}_{n+1}(E) &=&\mathbb{V}_{n}(E)\cup \mathbf{\wp }\left( \mathbb{V}%
_{n}(E)\right) ; \\
\mathbb{V}(E) &=&\bigcup\limits_{n=0}^{\infty }\mathbb{V}_{n}(E).
\end{eqnarray*}

If an object $x\in \mathbb{V}_{n+1}(E)\setminus \mathbb{V}_{n}(E)$ we say
that its rank is $n+1$, and we write $rank(x)=n+1$. With the usual
identifications of pairs with Kuratowski pairs and functions and relations
with their graphs, we have that $\mathbb{V}(E)$ contains all the usual
mathematical objects that can be constructed from $E$. Moreover, notice that
if $E$ is finite then also each finite level $V_{n}(E)$ of the
superstructure on $E$ is finite.

Now we let $\mathbb{A}$ be a set of atoms whose cardinality $\kappa $ is the
first strongly uncountable inaccessible cardinal number, and we assume that $%
\mathbb{R}\subset \mathbb{A}$. The mathematical universe we will consider in
this paper is 
\begin{equation*}
\Lambda =\left\{ E\in \mathbb{V}(\mathbb{A})\ |\ E\text{ is an atom or a set
such that}\ |E|<\kappa \right\}
\end{equation*}%
where $|E|$ denotes the cardinality of $E.$

We let $\mathfrak{L}$ be the family of finite subsets of $\Lambda $: 
\begin{equation*}
\mathfrak{L}=\mathbf{\wp }_{fin}(\Lambda ).
\end{equation*}%
$\mathfrak{L}$, ordered by the inclusion relation $\subseteq $, is a
directed set; if $E$ is any set, we call {\bfseries net} (with values in $E$%
) any function 
\begin{equation*}
\varphi :\mathfrak{L}\rightarrow E.
\end{equation*}%
From now on, we will denote by $\sqsubseteq $ a partial order relation over $%
\Lambda $ that extends the inclusion, namely such that $\forall \lambda ,\mu
\in \mathfrak{L,}$%
\begin{equation*}
\lambda \subseteq \mu \Rightarrow \lambda \sqsubseteq \mu.
\end{equation*}%
We assume that also $\left( \mathfrak{L},\sqsubseteq \right) $ is a directed
set; for the moment we will not make any other assumption on $E$. One of the
main task of this paper is to define $\sqsubseteq $ in such a way to get a
numerosity theory which satisfies the requests described in the introduction.

Let 
\begin{equation*}
\mathfrak{F}\left( \mathfrak{L},\mathbb{R}\right) =\left\{ \varphi \in 
\mathbb{R}^{\mathfrak{L}}\ |\ \exists A\in \Lambda ,\ \varphi \left( \lambda
\cap A\right) =\varphi \left( \lambda \right) \right\} \footnote{%
The choice of this particular space is due to the fact that we want to end
with the unique hyperreal field whose cardinality is the first inaccessible,
see \cite{keisler76}.}
\end{equation*}%
be endowed with the natural operations 
\begin{eqnarray*}
\left( \varphi +\psi \right) (\lambda ) &=&\varphi (\lambda )+\psi (\lambda
); \\
\left( \varphi \cdot \psi \right) (\lambda ) &=&\varphi (\lambda )\cdot \psi
(\lambda )
\end{eqnarray*}%
and the partial ordering%
\begin{equation*}
\varphi \geq \psi \Leftrightarrow \forall \lambda \in \mathfrak{L,\ }\varphi
(\lambda )\geq \psi (\lambda ).
\end{equation*}

The field of Euclidean numbers is defined as follows\footnote{%
This construction can be seen as an extension of $\alpha $-theory, see e.g. 
\cite{BDN2003alpha}.}:

\begin{definition}
\label{DR} The field of Euclidean numbers $\mathbb{E\supset R}$ is a field
so that there exists a surjective map 
\begin{equation*}
J:\mathfrak{F}\left( \mathfrak{L},\mathbb{R}\right) \rightarrow \mathbb{E}
\end{equation*}%
with the following properties:

\begin{enumerate}
\item[(i)] {\bfseries Ring homomorphism:} $J$ is a ring homomorphism,
namely for all $\varphi ,\psi \in \mathfrak{F}\left( \mathfrak{L},\mathbb{R}%
\right) $

\begin{itemize}
\item $J\left( \varphi +\psi \right) =J\left( \varphi \right) +J\left( \psi
\right) ;$

\item $J\left( \varphi \cdot \psi \right) =J\left( \varphi \right) \cdot
J\left( \psi \right) .$
\end{itemize}

\item[(ii)] \label{mon}{\bfseries Monotonicity:} for all $\varphi \in 
\mathfrak{F}\left( \mathfrak{L},\mathbb{R}\right) $, for all $r\in \mathbb{R}
$, if eventually $\varphi (\lambda )\geq r$ (namely there exists $\lambda
_{0}\in \mathfrak{L}$ such that $\forall \lambda \sqsupseteq \lambda _{0},\
\varphi (\lambda )\geq r$), then 
\begin{equation*}
J\left( \varphi \right) \geq r.
\end{equation*}
\end{enumerate}
\end{definition}

Let us show that such a field exists\footnote{%
Readers with a basic knowledge of nonstandard analysis will recognize
immediately that our construction is a minor modification of the usual limit
ultrapower construction.}.

\begin{proof} Let $\mathcal{U}$ be a fine ultrafilter on $\mathfrak{L}$,
namely a filter of sets such that

\begin{itemize}
\item {\bfseries Maximality:} $Q\in \mathcal{U}\Leftrightarrow \mathfrak{L}%
\backslash Q\notin \mathcal{U}$;

\item {\bfseries Finess:} $\forall \lambda \in \mathfrak{L},\ Q\left[
\lambda \right] \in \mathcal{U}$, where 
\begin{equation}
Q\left[ \lambda \right] :=\left\{ \mu \in \mathfrak{L}\ |\ \mu \sqsupseteq
\lambda \right\} .  \label{finess}
\end{equation}
\end{itemize}

The existence of $\mathcal{U}$ is a well known and easy consequence of
Zorn's Lemma. We use $\mathcal{U}$ to introduce an equivalence relation on
nets, by letting for all $\psi ,\varphi \in \mathfrak{F}\left( \mathfrak{B},%
\mathbb{R}\right) $ 
\begin{equation*}
\varphi \approx _{\mathcal{U}}\psi \Longleftrightarrow \exists Q\in \mathcal{%
U}\ \forall \lambda \in Q,\ \varphi \left( \lambda \right) =\psi \left(
\lambda \right) .
\end{equation*}

We set 
\begin{equation*}
\widetilde{\mathbb{E}}:=\mathfrak{F}\left( \mathfrak{L},\mathbb{R}\right)
/\approx _{\mathcal{U}}
\end{equation*}%
and we denote by $\left[ \varphi \right] _{\mathcal{U}}$ the equivalence
classes. Now we take a injective map 
\begin{equation*}
\Phi :\widetilde{\mathbb{E}}\rightarrow \mathbb{A}
\end{equation*}%
such that $\forall r\in \mathbb{R}$, 
\begin{equation*}
\Phi \left( \left[ c_{r}\right] _{\mathcal{U}}\right) =r
\end{equation*}%
were $c_{r}$ is the net identically equal to $r.$ Finally we set%
\begin{equation*}
\mathbb{E}=\Phi \left( \widetilde{\mathbb{E}}\right) .
\end{equation*}%
The operations on $\mathbb{E}$ can be easily defined by letting 
\begin{equation*}
\Phi \left( \left[ \varphi \right] _{\mathcal{U}}\right) +\Phi \left( \left[
\psi \right] _{\mathcal{U}}\right) =\Phi \left( \left[ \varphi +\psi \right]
_{\mathcal{U}}\right) ;\ \ \ \Phi \left( \left[ \varphi \right] _{\mathcal{U}%
}\right) \cdot \Phi \left( \left[ \psi \right] _{\mathcal{U}}\right) =\Phi
\left( \left[ \varphi \cdot \psi \right] _{\mathcal{U}}\right) .
\end{equation*}

It is very well known (see e.g. \cite{keisler76}) and simple to show that,
thanks to $\mathcal{U}$ being an ultrafilter, $\mathbb{E}$ endowed with the
above operations is a field; moreover, it can be made an ordered field by
endowing it with the following ordering: 
\begin{equation*}
\forall \varphi ,\psi \in \mathfrak{F}\left( \mathfrak{L},\mathbb{R}\right)
,\ \Phi \left( \left[ \varphi \right] _{\mathcal{U}}\right) \geq \Phi \left( %
\left[ \psi \right] _{\mathcal{U}}\right) :\Longleftrightarrow \exists Q\in 
\mathcal{U},\ \forall \lambda \in Q\ \ \varphi \left( \lambda \right) \geq
\psi \left( \lambda \right) .\qedhere
\end{equation*}
\end{proof}

\begin{remark}
$\mathbb{E}$ is an hyperreal field whose cardinality is $\kappa $; such a
field is unique up to isomorphisms (see \cite{keisler76}); namely, changing "%
$\sqsubseteq $" we get an isomorphic hyperreal field. However, we will
choose "$\sqsubseteq $" in such a way to get interesting interactions with
other mathematical structures.
\end{remark}

The number $J\left( \varphi \right) $ is called the $\Lambda$-limit of the
net $\varphi $ and will be denoted by%
\begin{equation*}
J\left( \varphi \right) =\lim_{\lambda \uparrow \Lambda }\varphi (\lambda ).
\end{equation*}%
The reason of this name and notation is that the operation%
\begin{equation*}
\varphi \mapsto \lim_{\lambda \uparrow \Lambda }\varphi (\lambda )
\end{equation*}%
satisfies many of the properties of the usual Cauchy limit, but with the
stronger property of existing for every net. More exactly, it satisfies the
following properties:

\begin{itemize}
\item \textbf{Existence:} Every net $\varphi :\mathfrak{L}\rightarrow 
\mathbb{R}$ has a unique limit $L\in \mathbb{E}$.

\item \textbf{Monotonicity:} For all $r\in \mathbb{R}$ if eventually $%
\varphi (\lambda )\geq r$, then%
\begin{equation*}
\lim_{\lambda \uparrow \Lambda }\varphi (\lambda )\geq r;
\end{equation*}

\item \textbf{Sum and product:} For all $\varphi ,\psi :\mathfrak{L}%
\rightarrow \mathbb{R}$ 
\begin{eqnarray*}
\lim_{\lambda \uparrow \Lambda }\varphi (\lambda )+\lim_{\lambda \uparrow
\Lambda }\psi (\lambda ) &=&\lim_{\lambda \uparrow \Lambda }\left( \varphi
(\lambda )+\psi (\lambda )\right) , \\
\lim_{\lambda \uparrow \Lambda }\varphi (\lambda )\cdot \lim_{\lambda
\uparrow \Lambda }\psi (\lambda ) &=&\lim_{\lambda \uparrow \Lambda }\left(
\varphi (\lambda )\cdot \psi (\lambda )\right) .
\end{eqnarray*}
\end{itemize}

Notice that, if $\lim_{\lambda\rightarrow\Lambda}\varphi(\lambda)$ denotes
the usual Cauchy limit of $\varphi$, the relationship between the Cauchy
limit and the $\Lambda$-limit is

\begin{equation*}
\lim_{\lambda\rightarrow\Lambda}\varphi(\lambda)=st\left(\lim_{\lambda%
\uparrow\Lambda}\varphi(\lambda)\right).
\end{equation*}

\begin{remark}
\label{figa}The notion of Eucliean field defined by Definition \ref{DR} has
been used in several papers with "$\subseteq $" instead of "$\sqsubseteq $"
(e.g. \cite{ultra}, \cite{milano}, \cite{bls}). Now, we will explain the
main technical reason for using an Euclidean field rather than a "generic"
hyperreal field. A set $F\subset \Lambda ^{\ast }$ is called hyperfinite if 
\begin{equation*}
F=\lim_{\lambda \uparrow \Lambda }F_{\lambda }=\left\{ \lim_{\lambda
\uparrow \Lambda }\ \varphi (\lambda )\ |\ \varphi (\lambda )\in F_{\lambda
}\right\}
\end{equation*}%
where the sets $F_{\lambda }\in \Lambda $ are finite. Hyperfinite sets play
a crucial role in many applications of nonstandard analysis. If we use a
Euclidean field, we can associate to every set $E\in \Lambda $ a \textbf{%
unique} hyperfinite set $E^{\circledast }$ defined as follows 
\begin{equation*}
E^{\circledast }=\lim_{\lambda \uparrow \Lambda }E\cap \lambda =\left\{
\lim_{\lambda \uparrow \Lambda }\ \varphi (\lambda )\ |\ \varphi (\lambda
)\in E\cap \lambda \right\}.
\end{equation*}%
The set $E^{\circledast }$ satisfies the property\footnote{%
Here, as usual, we have set%
\begin{equation*}
E^{\sigma }=\left\{ x^{\ast }\ |\ x\in E\right\} .
\end{equation*}%
} 
\begin{equation*}
E^{\sigma }\subset E^{\circledast }\subset E^{\ast }
\end{equation*}%
which is very useful in the applications. Moreover, using an Euclidean field
we can easily define the numerosity function over any set $E\in \Lambda $ by
setting (see Section \ref{NN})%
\begin{equation}
\mathfrak{num}\left( E\right) :=\lim_{\lambda \uparrow \Lambda }\ \left\vert
E\cap \lambda \right\vert .  \label{NT+}
\end{equation}
In this paper, replacing with "$\subseteq $" with a suitable "$\sqsubseteq $%
", the numerosity theory given by Equation (\ref{NT+}) is consistent with
the main features of the numerosity theories present in the litterature
(e.g. \cite{benci95b}, \cite{BDN2018}, \cite{BDN2003}, \cite{BDNF1}, \cite%
{BDNB1}, \cite{BDNB}, \cite{BF}, \cite{QSU}, \cite{DNFtup}, \cite{FM}, \cite%
{mancu}).
\end{remark}

\subsection{Labelled sets}

The notion of labelled set has been introduced in \cite{benci95b} and \cite%
{BDN2003} to construct a numerosity theory for countable sets. Here we
extend this notion to adapt it to the study of numerosity theories for
larger sets.

\begin{definition}
\label{LB}We call \textbf{label set} a family of sets $\mathfrak{B}\subset 
\mathfrak{L}$ such that

\begin{itemize}
\item[(i)] \label{bbbc} $\forall \mathfrak{s,t}\in \mathfrak{B,}\ \mathfrak{s%
}\cap \mathfrak{t},\ \mathfrak{s}\cup \mathfrak{t}\in \mathfrak{B} $;

\item[(ii)] \label{bbbb} $\forall \mathfrak{s}\in \mathfrak{B},\ \mathfrak{s}%
\cap \mathfrak{L}=\varnothing ;$

\item[(iii)] \label{bbbd} $\bigcup_{\mathfrak{s}\in \mathfrak{B}}\mathbb{V}%
\left( \mathfrak{s}\right) =\Lambda $.
\end{itemize}
\end{definition}

Requirement (i) gives to $\mathfrak{B}$ a lattice structure, whilst
requirement (ii) entails that the elements of a label are either atoms or
infinite sets.

Having fixed the notion of "labels", we can now introduce the notion of
"labelling":

\begin{definition}
\label{labling}Let $\mathfrak{B}$ be a set of labels. We call {$\mathfrak{B}$%
-}\textbf{labelling} the map 
\begin{equation*}
\ell :\Lambda \rightarrow \mathfrak{B}
\end{equation*}
defined as follows:%
\begin{equation*}
\ell (a)=\bigcap_{\mu \in I_{a}}\mu ,
\end{equation*}%
where $I_{a}=\{\mu \in \mathfrak{B}\mid a\in \mathbb{V}(\mu )\}$. For every $%
a\in \mathfrak{L}$ we call $\ell (a)$ the label of $a$.
\end{definition}

Roughly speaking, the label of an object $a\in \Lambda $ is a finite set
whose elements allow to define $a$ by the fundamental finitistic set
operations.

There is plenty of sets of labels: just set 
\begin{equation}
\mathfrak{B}_{\max }=\left\{ \mathfrak{t}\in \mathfrak{L}\mid \ \mathfrak{t}%
\cap \mathfrak{L}=\emptyset \right\} .  \label{bmax}
\end{equation}%
Obviously, $\mathfrak{B}_{\max }$ is a label set; more importantly, every
label set $\mathfrak{B}$ is a subset of $\mathfrak{B}_{\max }$.

\begin{example}
Let $a=\mathbb{N}$. Then, using the $\mathfrak{B}_{\max }$-labelling, 
\begin{equation*}
\ell \left( \mathbb{N}\right) =\{\mathbb{N}\}.
\end{equation*}
\end{example}

Now we will describe some properties of a set of labels {$\mathfrak{B}$} and
the corresponding $\mathfrak{B}$-labelling:

\begin{proposition}
\label{eccola} Let $\mathfrak{B}$ be a set of labels, and let $\mathfrak{s}%
\in \mathfrak{B}$. Then

\begin{enumerate}
\item[$(i)$] $\mathbb{V}(\mathfrak{s})$ is countable;

\item[$(ii)$] $\mathbb{V}(\mathfrak{s})\setminus \mathfrak{s}$ consists only
of finite sets;

\item[$(iii)$] \label{cor 1}for all $\mathfrak{s},\mathfrak{t}\in \mathfrak{B%
},$ for all $m\in \mathbb{N}$ we have that 
\begin{equation*}
\mathbb{V}_{m}(\mathfrak{s})\subseteq \mathbb{V}_{m}(\mathfrak{t}%
)\Leftrightarrow \mathfrak{s}\subseteq \mathfrak{t};
\end{equation*}

\item[$(iv)$] \label{c5} for all $\mathfrak{s},\mathfrak{t}\in \mathfrak{B},$
$\mathbb{V}(\mathfrak{s}\cap \mathfrak{t})=\mathbb{V}(\mathfrak{s})\cap 
\mathbb{V}(\mathfrak{t})$;

\item[$(v)$] \label{c6}If $a\in \mathfrak{L}$ and $\mathfrak{t}\in \mathfrak{%
B}\mathfrak{,\ }$then $\left\{ a\right\} \in \mathbb{V}(\mathfrak{t}%
)\Leftrightarrow a\in \mathbb{V}(\mathfrak{t})$.
\end{enumerate}
\end{proposition}

\begin{proof} (i) $\mathfrak{s}$ is finite, hence by induction it
trivially holds that $\mathbb{V}_{n}\left( \mathfrak{s}\right) $ is finite
for every $n\in \mathbb{N}$. Therefore $\mathbb{V}\left( \mathfrak{s}\right) 
$ is countable.

(ii) As $\mathfrak{s}$ is finite, by induction it is immediate to prove that $%
\mathbb{V}_{n+1}(\mathfrak{s})\setminus \mathfrak{s}$ consists only of
finite sets, hence the thesis follows straightforwardly.

(iii) The implication $\Leftarrow $ is trivial. Let us prove the
other implication. If $\mathbb{V}_{m}(\mathfrak{s})\subseteq \mathbb{V}_{m}(%
\mathfrak{t})$ then, in particular, $\mathfrak{s}\in \mathbb{V}_{m}(%
\mathfrak{t})$. Now, if $\mathfrak{s}=\{a_{1},\dots ,a_{n}\}$, all $%
a_{1},\dots ,a_{n}$ are either atoms or infinite sets, hence by (i), $%
a_{1},\dots ,a_{n}\in \mathfrak{t}$.

(iv) The inclusion $\subseteq $ is trivial. For the reverse inclusion,
let $\eta \in \mathbb{V}(\mathfrak{s})\cap \mathbb{V}(\mathfrak{t})$. In
particular, $\eta \in \mathbb{V}_{n}(\mathfrak{s})\cap \mathbb{V}_{m}(%
\mathfrak{t})$ and so, if $l=\max \{n,m\}$, we have that $\eta \in \mathbb{V}%
_{l}(\mathfrak{s})\cap \mathbb{V}_{l}(\mathfrak{t})$. We proceed by
induction on $l$ to show that $\mathbb{V}_{l}(\mathfrak{s})\cap \mathbb{V}%
_{l}(\mathfrak{t})\subseteq \mathbb{V}(\mathfrak{s}\cap \mathfrak{t})$.

If $l=0$, then $\eta \in \mathbb{V}_{0}(\mathfrak{s})\cap \mathbb{V}_{0}(%
\mathfrak{t})$ if and only if $\eta =\mathfrak{s}=\mathfrak{t}$, and the
desired inclusion trivially holds.

Now let us suppose the inclusion to hold for $l\in \mathbb{N}$, and let $%
\eta \in \mathbb{V}_{l+1}(\mathfrak{s})\cap \mathbb{V}_{l+1}(\mathfrak{t})$.
If $\eta \in \mathbb{V}_{l}(\mathfrak{s})\cap \mathbb{V}_{l}(\mathfrak{t})$
we are done by inductive hypothesis; if not, there are $A\in \mathbb{V}_{l}(%
\mathfrak{s}),B\in \mathbb{V}_{l}(\mathfrak{t})$ such that $\eta \in \wp
(A)\cap \wp (B)$. In particular, $\eta \in \wp (A\cap B)$. But $A\cap B\in 
\mathbb{V}_{l}(\mathfrak{s})\cap \mathbb{V}_{l}(\mathfrak{t})$, so by
induction $A\cap B\in \mathbb{V}(\mathfrak{s}\cap \mathfrak{t})$, hence $%
\eta \in \mathbb{V}(\mathfrak{s}\cap \mathfrak{t})$ as desired.

(v) The implication $\Leftarrow $ is trivial. For the reverse
implication, let 
\begin{equation*}
l=\min \{n\in \mathbb{N}\mid a\in \mathbb{V}_{n}(\mathfrak{t})\}.
\end{equation*}%
In particular, we have that $l\geq 1$. In fact, if $l=0$ then $\{a\}=%
\mathfrak{t}$, and this cannot happen as $\mathfrak{t}\cap \mathfrak{L}%
=\emptyset $. Hence $a\in \mathbb{V}_{l-1}(\mathfrak{t})$ and we are done.\end{proof}

As we will see in Section \ref{NUOVA}, the freedom of choosing a particular
set of labels allows to impose certain additional arithmetical properties on
numerosities.

\begin{proposition}
\label{pro} Let $\mathfrak{B}$ be a set of labels and let $\ell $ be a $%
\mathfrak{B}$-labelling. The following properties hold:

\begin{enumerate}
\item[$(i)$] \label{v1}$\forall a,b\in \Lambda $, $a\subseteq b\Rightarrow
\ell (a)\subseteq \ell (b);$

\item[$(ii)$] \label{v2}$\forall a\in \Lambda ,$ $\ell (a)\supseteq\ell
(\{a\}),$ and equality holds if $a\in\mathfrak{L}$;

\item[$(iii)$] \label{v3}$\forall a,b\in \Lambda ,$ $a\in b\Rightarrow \ell
(\{a\})\subseteq \ell (b);$

\item[$(iv)$] \label{v5}$\forall \lambda \in \mathfrak{B},$ $\lambda \in 
\mathbb{V}\left( \ell (\lambda )\right) ;$

\item[$(v)$] \label{labelling 1}$\forall \mathfrak{s}\in \mathfrak{B},\ \ell
\left( \mathfrak{s}\right) =\mathfrak{s};$

\item[$(vi)$] \label{v4}$\forall \mathfrak{s}\in \mathfrak{B},\forall m\in 
\mathbb{N},$ $\ell \left( \mathbb{V}_{m}(\mathfrak{s})\right) =\mathfrak{s};$

\item[$(vii)$] \label{labelling 2}$\forall a\in \Lambda ,\forall \mathfrak{s}%
\in \mathfrak{B},$ $\ell \left( a\right) \subseteq \mathfrak{s}%
\Leftrightarrow a\in \mathbb{V}(\mathfrak{s})$;

\item[$(viii)$] \label{labelling 3}$\forall a,b\in \Lambda ,$ $\ell \left(
\{a,b\}\right) =\ell \left( a\right) \cup \ell (b)$;

\item[$(ix)$] \label{l1}$\forall a,b\in \Lambda ,$ $\ell \left( \left(
a,b\right) \right) =\ell (a)\cup \ell (b)$.

\item[$(x)$] \label{l3}$\forall E\in \Lambda $, $\forall \lambda \in 
\mathfrak{L}$ if we set 
\begin{equation*}
E_{\lambda }:=\left\{ x\in E\ |\ \ell \left( x\right) \subseteq \ell \left(
\lambda \right) \right\}
\end{equation*}%
then%
\begin{equation*}
E_{\lambda }=E\cap \mathbb{V}(\ell \left( \lambda \right) );
\end{equation*}

\item[$(xi)$] \label{l4}$\forall E\in \Lambda ,$ the set $E_{\lambda }$ is
finite.
\end{enumerate}
\end{proposition}

\begin{proof} (i) If $a\subseteq b,$ then trivially $a\in \mathbb{V}(%
\mathfrak{t})$ whenever $b\in \mathbb{V}(\mathfrak{t})$ and hence 
\begin{equation*}
\ell (a)=\bigcap \left\{ \mathbb{V}(\mathfrak{t})\ |\ \mathfrak{t}\in 
\mathfrak{B},\ a\in \mathbb{V}(\mathfrak{t})\right\} \subseteq \bigcap
\left\{ \mathbb{V}(\mathfrak{t})\ |\ \mathfrak{t}\in \mathfrak{B},\ b\in 
\mathbb{V}(\mathfrak{t})\right\} =\ell (b).
\end{equation*}

(ii) Trivially, for all $\mathfrak{t}\in\mathfrak{B}$, if $a\in\mathbb{V}(\mathfrak{t})$ then also $\{a\}\in\mathbb{V}(\mathfrak{t})$, hence $\ell(\{a\})\subseteq \ell(\{a\}$. The second claim follows from the fact that, by Proposition \ref{eccola}.(v), $a\in \mathbb{V}(\mathfrak{t})\Leftrightarrow \{a\}\in \mathbb{V}(%
\mathfrak{t})$.

(iii) If $a\in b,$ then $\left\{ a\right\} \subseteq b$ and by (i) and (ii), we have that%
\begin{equation*}
\ell (a)=\ell (\left\{ a\right\} )\subseteq \ell (b).
\end{equation*}

(iv) By definition, $\forall \mu \in I_{a},$ $a\in \mathbb{V}(\mu );$
hence 
\begin{equation*}
a\in \bigcap_{\mu \in I_{a}}\mathbb{V}(\mu )=\mathbb{V}\left( \bigcap_{\mu
\in I_{a}}\mu \right) =\mathbb{V}\left( \ell (a)\right) .
\end{equation*}

(v) We have that $\mathfrak{s}=\mathbb{V}_{0}(\mathfrak{s}%
)\in \mathbb{V}(\mathfrak{s});$ hence $\mathfrak{s}\in I_{\mathfrak{s}}$ and
so $\ell (\mathfrak{s})\subseteq \mathfrak{s}.$ Moreover, if $\mathfrak{t}%
\in I_{\mathfrak{s}}$, $\mathfrak{s}\in \mathbb{V}(\mathfrak{t})$ and since $%
\mathfrak{s}$ is a label, $\mathfrak{s}\subset \mathfrak{t}$ and so 
\begin{equation*}
\mathfrak{s}\subseteq \bigcap_{\mathfrak{t}\in I_{\mathfrak{s}}}\mathfrak{t}%
=\ell (\mathfrak{s}).
\end{equation*}

(vi) If $\mathfrak{s}\in \mathfrak{B},$ then $\forall \mathfrak{t}\in 
\mathfrak{B},$ by Proposition \ref{eccola}.(iv) we have that%
\begin{equation*}
\mathbb{V}_{m}(\mathfrak{s})\in \mathbb{V}(\mathfrak{t})\Leftrightarrow 
\mathfrak{s}\subseteq \mathfrak{t}\Leftrightarrow \mathfrak{s}\in \mathbb{V}(%
\mathfrak{t}).
\end{equation*}

Then%
\begin{eqnarray}
I_{\mathbb{V}_{m}(\mathfrak{s})} &=&\{\mathfrak{t}\in \mathfrak{B}\mid 
\mathbb{V}_{m}(\mathfrak{s})\in \mathbb{V}(\mathfrak{t})\}=\{\mathfrak{t}\in 
\mathfrak{B}\mid \mathfrak{s}\subseteq \mathfrak{t}\}  \label{v41} \\
&=&\{\mathfrak{t}\in \mathfrak{B}\mid \mathfrak{s}\in \mathbb{V}(\mathfrak{t}%
)\}=I_{\mathfrak{s}},  \notag
\end{eqnarray}

hence the thesis follows by the definition of $\mathfrak{B}$-labeling.

(vii) ($\Rightarrow $) If $\ell \left( a\right) \subseteq 
\mathfrak{s}$, then, by (i), $\mathbb{V}\left( \ell \left( a\right)
\right) \subseteq \mathbb{V}\left( \mathfrak{s}\right) ;$ therefore, by (iv),$\ a\in \mathbb{V}\left( \mathfrak{s}\right) .$

($\Leftarrow $) If $a\in \mathbb{V}(\mathfrak{s}),$ then $\mathfrak{s}\in I_{a}$ and so $\ell(a)\subseteq\mathfrak{s}$.

(viii) For $\mathfrak{t}\in I_{\left\{ a,b\right\} },$ since $%
\left\{ a,b\right\} \in \mathbb{V}(\mathfrak{t})$, but $\left\{ a,b\right\}
\notin \mathfrak{t}$, we have that $a\in \mathbb{V}(\mathfrak{t})$ and $b\in 
\mathbb{V}(\mathfrak{t});$ then%
\begin{eqnarray*}
I_{\left\{ a,b\right\} } &=&\{\mathfrak{t}\in \mathfrak{B}\mid a\in \mathbb{V%
}(\mathfrak{t})\ \text{and\ }b\in \mathbb{V}(\mathfrak{t})\} \\
&=&\{\mathfrak{t}\in \mathfrak{B}\mid a\in \mathbb{V}(\mathfrak{t})\}\cap \{%
\mathfrak{t}\in \mathfrak{B}\mid b\in \mathbb{V}(\mathfrak{t})\}=I_{a}\cap
I_{b}.
\end{eqnarray*}%
Hence%
\begin{eqnarray*}
\ell (\left\{ a,b\right\} ) &=&\bigcap_{\mathfrak{t}\in I_{\left\{
a,b\right\} }}\mathfrak{t}=\bigcap_{\mathfrak{t}\in I_{a}\cap I_{b}}%
\mathfrak{t}=\{x\in \Lambda \mid x\in \mathfrak{t}\ \ \text{and}\ \ 
\mathfrak{t}\in I_{a}\cap I_{b}\} \\
&=&\{x\in \Lambda \mid (x\in \mathfrak{t}\ \text{and}\ \mathfrak{t}\in
I_{a})\ \text{or\ }(x\in \mathfrak{t}\ \text{and}\ \mathfrak{t}\in I_{b})\}
\\
&=&\{x\in \Lambda \mid (x\in \mathfrak{t}\ \text{and}\ \mathfrak{t}\in
I_{a})\}\cup \{x\in \Lambda \mid (x\in \mathfrak{t}\ \text{and}\ \mathfrak{t}%
\in I_{b})\} \\
&=&\left( \bigcap_{\mathfrak{t}\in I_{a}}\mathfrak{t}\right) \cup \left(
\bigcap_{\mathfrak{t}\in I_{b}}\mathfrak{t}\right) =\ell (a)\cup \ell (b).
\end{eqnarray*}

(ix) We have that%
\begin{equation*}
\ell \left( \left( a,b\right) \right) =\ell \left( \{a,\{a,b\}\}\right)
=\ell \left( a\right) \cup \ell \left( \{a,b\}\right) =\ell \left( b\right)
\ell \left( a\right) \cup \ell \left( a\right) \cup \ell \left( b\right)
=\ell \left( a\right) \cup \ell \left( b\right) .
\end{equation*}

(x) First set $\mathfrak{s=\ell }\left( \lambda \right) .$ Let us
first prove the inclusion $\subseteq $. Let $x\in E$ be such that $\ell
(x)\subseteq \mathfrak{s}$. By the definition of labelling then $x\in 
\mathbb{V}(\mathfrak{s})$, and the inclusion is proven. For the reverse
inclusion, let $x\in E\cap \mathbb{V}(\mathfrak{s})$. In particular, it must
be $\ell (x)\subseteq \mathfrak{s}$, and we are done.

(xi) If $E\in \Lambda $ then $E$ has a finite rank $n$, which means
that $E\cap \mathbb{V}(\mathfrak{s})=E\cap \mathbb{V}_{n}(\mathfrak{s})$,
and the conclusion follows by (x) as, by construction, $\mathbb{V}%
_{n}(\mathfrak{s})$ is finite. \end{proof}

The notion of $\mathfrak{B}$-labelling allows to equip $\mathfrak{L}$ with a
partial order structure $\sqsubseteq $:

\begin{definition}
We set 
\begin{equation*}
\mathfrak{L}_{0}\left(\mathfrak{B}\right):=\{\mathbb{V}_{m}(\mathfrak{t})\
|\ \ m\in \mathbb{N}_{0},\ \mathfrak{t}\in \mathfrak{B}\}
\end{equation*}%
and for every $\lambda ,\mu \in \mathfrak{L}$, we set%
\begin{equation*}
\lambda \sqsubseteq \mu \Leftrightarrow \lambda \subseteq \dbigcap \left\{
\tau \in \mathfrak{L}_{0}\mathfrak{\ }|\ \mu \subseteq \tau \right\}.
\end{equation*}
\end{definition}

Notice that, by definition 
\begin{equation*}
\lambda \sqsubseteq \mu \Leftrightarrow I_{\lambda}\subseteq I_{\mu},
\end{equation*}
where $I_{\lambda}$ has been introduced in Definition \ref{labling}, and
that 
\begin{equation*}
\mathfrak{L}_{0}\left(\mathfrak{L_{0}}\left(\mathfrak{B}\right)\right)=%
\mathfrak{L}_{0}\left(\mathfrak{B}\right).
\end{equation*}

Clearly $\sqsubseteq $ induces a lattice structure on $\mathfrak{L}_{0}\left(%
\mathfrak{B}\right)$, since 
\begin{equation*}
\lambda \vee \mu :=\dbigcap \left\{ \tau \in \mathfrak{L}_{0}\mathfrak{\ }|\
\lambda \cup \mu \subseteq \tau \right\} ;
\end{equation*}%
\begin{equation*}
\lambda \wedge \mu :=\dbigcup \left\{ \tau \in \mathfrak{L}_{0}\mathfrak{\ }%
|\ \tau \subseteq \lambda \cap \mu \right\} .
\end{equation*}%
Since $\lambda \subseteq \mu \Rightarrow \lambda \sqsubseteq \mu ,$ we can
use the directed set $\left( \mathfrak{L}_{0}\left(\mathfrak{B}%
\right),\sqsubseteq \right) $ to define a field of Euclidean numbers as in
Definition \ref{DR}. From now on, $\mathbb{E}$ will denote such a field.

\begin{remark}
Now the idea is to construct a suitable set of labels in such a way that the
relation $\sqsubseteq $ carry all the informations needed for a "good"
numerosity theory. All these informations depend on $\sqsubseteq $ and not
on the choice of the ultrafilter used in the construction of $\mathbb{E}$.
\end{remark}

\section{The general theory of numerosities\label{NN}}

Different versions of the notion of numerosity have already been studied in
several previous papers \cite{benci95b,BDN2003, BDNF1, BDNB1, BDNB, QSU,
DNFtup, FM}; we refer also to the book \cite{BDN2018} for a complete
overview of the countable case. In this paper, we want to show how the new
definition of labels and of the Euclidean field allows to easily provide the
most interesting features of the theory of numerosities. In particular, we
show how numerosities can be used to simultaneously unify and generalize
objects and results coming from different areas, like (a version of)
Lagrange's Theorem for groups, the Peano-Jordan measure and the Lebesgue
measure.

\subsection{Definition and first properties}

\begin{definition}
\label{NT}Let $E$ be a set in $\Lambda $. We call {\textbf{numerosity} of $E$%
} the Euclidean number%
\begin{equation*}
\mathfrak{num}\left( E\right) :=\lim_{\lambda \uparrow \Lambda }\ \left\vert
E\cap \lambda \right\vert .
\end{equation*}%
The set of numerosities will be denoted by $\mathbf{Num}$.
\end{definition}

The notion of numerosity allows to "give a name" to some hyperreal number.
We set 
\begin{equation}
\alpha =\mathfrak{num}\left( \mathbb{N}\right) ;\ \beta =\mathfrak{num}%
\left( \left[ 0,1\right) \right).  \label{alfa}
\end{equation}

The numerosity of a set depends on the choice of the set of labels $%
\mathfrak{B}$, as well as on the ultrafilter $\mathcal{U}$ on $\mathfrak{B}$
chosen to construct $\mathbb{E}$. However the properties which will be
listed below are independent of any choice.

\begin{theorem}
\label{1} Let $E,F$ be sets in $\Lambda $. Numerosities satisfy the
following properties:

\begin{enumerate}
\item[$(i)$] \textbf{Finite sets principle:} if $E$ is a finite set, then $%
\mathfrak{num}\left( E\right) =|E|$;

\item[$(ii)$] \textbf{Euclid's principle:} if $E\subset F$ then $\mathfrak{%
num}\left( E\right) <\mathfrak{num}\left( F\right)$;

\item[$(iii)$] \textbf{Labels principle}: if 
\begin{equation*}
E_{\lambda }=\left\{ x\in E\ |\ \ell \left( x\right) \subseteq \ell \left(
\lambda \right) \right\} .
\end{equation*}%
then, if $\lambda \in \mathfrak{L}_{0}\left(\mathfrak{B}\right)$, $%
E_{\lambda }=E\cap \lambda $ and hence 
\begin{equation*}
\mathfrak{num}\left( E\right) =\lim_{\lambda \uparrow \Lambda }\ \left\vert
E_{\lambda }\right\vert;
\end{equation*}

\item[$(iv)$] \textbf{Comparison principle:} if $\Phi :E\rightarrow F$ is a
bijection that preserves labels, namely such that for all $x\in E$ 
\begin{equation*}
\ell (\Phi \left( x\right) )=\ell \left( x\right) ,
\end{equation*}%
then $\mathfrak{num}\left( E\right) =\mathfrak{num}\left( F\right) $;

\item[$(v)$] \textbf{Sum principle:} if $E\cap F=\varnothing $ then $%
\mathfrak{num}\left( E\cup F\right) =\mathfrak{num}\left( E\right) +%
\mathfrak{num}\left( F\right) ;$

\item[$(vi)$] \textbf{Product principle:} $\mathfrak{num}\left( E\times
F\right) =\mathfrak{num}\left( E\right) \cdot \mathfrak{num}\left( F\right) $%
;

\item[$(vii)$] \textbf{Finite parts principle:} $\mathfrak{num}\left( 
\mathbf{\wp }_{fin}\left( E\right) \right) =2^{\mathfrak{num}\left( E\right)
}$;

\item[$(viii)$] \textbf{Finite functions principle:} let $E$ be nonempty,
and 
\begin{equation*}
\mathfrak{F}_{fin}\left( X,E\right) :=\left\{ f:D\rightarrow E\mid D\in 
\mathbf{\wp }_{fin}(X)\right\} .
\end{equation*}%
Then, if $a\in E$, we have 
\begin{equation*}
\mathfrak{num}\left( \mathfrak{F}_{fin}\left( X,E\backslash \left\{
a\right\} \right) \right) =\mathfrak{num}\left( E\right) ^{\mathfrak{num}%
(X)}.
\end{equation*}
\end{enumerate}
\end{theorem}

\begin{proof} (i) If $|E|=n<\infty $, then for every $\lambda \in 
\mathfrak{L}$, we have $\left\vert E\cap \lambda \right\vert =n$, and the
thesis then follows by taking the $\lambda $-limit.

(ii) If $E\subset F$, eventually $\left\vert E\cap \lambda \right\vert
<\left\vert F\cap \lambda \right\vert $, so $\lim_{\lambda \uparrow \Lambda
}\left\vert E\cap \lambda \right\vert <\lim_{\lambda \uparrow \Lambda
}\left\vert E\cap \lambda \right\vert $.

(iii) Take $\lambda =\mathbb{V}_{m}(\mathfrak{s})$ with $m\geq rank(E)$
and $\mathfrak{s\in B}$. Then, by Proposition \ref{pro}.\ref{l3}%
\begin{equation*}
E\cap \lambda =E\cap \mathbb{V}_{m}(\mathfrak{s})=E\cap \mathbb{V}(\mathfrak{%
s})=E_{\lambda }.
\end{equation*}

(iv) By hypothesis we have that for all $\lambda \in \mathfrak{L}$ $%
|E_{\lambda }|=|F_{\lambda }|$, and so by the labels principle 
\begin{equation}
\mathfrak{num}\left( E\right) =\lim_{\lambda \uparrow \Lambda }\ \left\vert
E\cap \lambda \right\vert =\lim_{\lambda \uparrow \Lambda }\ \left\vert
F\cap \lambda \right\vert =\mathfrak{num}\left( F\right) .  \label{nina}
\end{equation}

(v) Just notice that $|E\cup F|_{\lambda }=|E_{\lambda }|+|F_{\lambda
}|$ for every $\lambda \in \mathfrak{L}$, hence the thesis follows by
Definition \ref{DR}.(2) and by the labels principle.

(vi) Let $\lambda \in \mathfrak{L}$. By property (ix) in
Proposition \ref{pro}, we have that $\left( E\times F\right) _{\lambda
}=E_{\lambda }\times F_{\lambda }$, hence $|\left( E\times F\right)
_{\lambda }|=|E_{\lambda }\times F_{\lambda }|=|E_{\lambda }|\cdot
|F_{\lambda }|$, and the thesis then follows immediately, again by the labels principle.

(vii) Let $\lambda =\mathbb{V}_{m}(\mathfrak{s})\in \mathfrak{L}_{0}$ ($%
m>rank(E))$, and let $a\in \mathbf{\wp }_{fin}\left( E\right) \cap \mathbb{V}(%
\mathfrak{s})$. Then by Proposition \ref{pro}.(x) we have that it must be $%
a\in \mathbf{\wp }_{fin}\left( E_{\lambda }\right) $. Conversely, if $a\in 
\mathbf{\wp }_{fin}\left( E_{\lambda }\right) $ it is immediate to see that $%
a\in \mathbf{\wp }_{fin}\left( E\right) \cap \mathbb{V}(\mathfrak{s})$.
Hence, by Proposition \ref{pro} we have 
\begin{equation*}
\left\vert \left[ \mathbf{\wp }_{fin}\left( E\right) \right] _{\lambda
}\right\vert =\left\vert \mathbf{\wp }_{fin}\left( E\right) \cap \mathbb{V}(%
\mathfrak{s})\right\vert =\left\vert \mathbf{\wp }_{fin}\left( E_{\lambda
}\right) \right\vert =2^{\left\vert E_{\lambda }\right\vert },
\end{equation*}%
and so by the labels principle 
\begin{equation*}
\mathfrak{num}\left( \mathbf{\wp }_{fin}\left( E\right) \right)
=\lim_{\lambda \uparrow \Lambda }\ 2^{\left\vert E_{\lambda }\right\vert
}=2^{\mathfrak{num}\left( E\right) }.
\end{equation*}

(viii) We set $\lambda =\mathbb{V}_{m}(\mathfrak{s})\in \mathfrak{L}_{0}$%
, $m>rank(f).$ Let $f\in \mathfrak{F}_{fin}\left( X,E\backslash \{a\}\right)
\cap \mathbb{V}(\mathfrak{s})$, and let $D$ be the domain of $f$. By
identifying functions with Kuratowski pairs, and by our definition of
labellings on pairs, it is immediate to see that $f\in \mathfrak{F}%
_{fin}\left( X,E\backslash \{a\}\right) \cap \mathbb{V}(\mathfrak{s})$ if
and only if $D\left( f\right) \subset X\cap \mathbb{V}(\mathfrak{s}%
)=X_{\lambda }$ and $Im\left( f\right) \subset \left( E\backslash
\{a\}\right) \cap \mathbb{V}(\mathfrak{s})=E_{\lambda }\backslash \{a\}.$
Therefore 
\begin{equation*}
\mathfrak{F}_{fin}\left( X,E\backslash \{a\}\right) \cap \mathbb{V}(%
\mathfrak{s})=\mathfrak{F}_{fin}\left( X_{\lambda },E_{\lambda }\backslash
\{a\}\right) .
\end{equation*}

Notice that 
\begin{equation*}
\left\vert \mathfrak{F}_{fin}\left( X_{\lambda },E_{\lambda }\backslash
\{a\}\right) \right\vert =\left\vert \mathfrak{F}\left( X_{\lambda
},E_{\lambda }\right) \right\vert .
\end{equation*}%
In fact, the association $g\in \mathfrak{F}_{fin}\left( X_{\lambda
},E_{\lambda }\backslash \{a\}\right) \rightarrow \widetilde{g}\in \mathfrak{%
F}\left( X_{\lambda },E_{\lambda }\right) $, with 
\begin{equation*}
\widetilde{g}(x)=%
\begin{cases}
g(x), & \text{if}\ x\in X_{\lambda }; \\ 
a, & \text{otherwise}%
\end{cases}%
\end{equation*}%
is a bijection. Hence, again by the labels principle, 
\begin{eqnarray*}
\mathfrak{num}\left( \mathfrak{F}_{fin}\left( X,E\backslash \{a\}\right)
\right) &=&\lim_{\lambda \uparrow \Lambda }\ \left\vert \mathfrak{F}%
_{fin}\left( X,E\backslash \{a\}\right) \cap \mathbb{V}(\lambda )\right\vert
=\lim_{\lambda \uparrow \Lambda }\ \left\vert \mathfrak{F}\left( X_{\lambda
},E_{\lambda }\right) \right\vert \\
&=&\lim_{\lambda \uparrow \Lambda }\ \left\vert E_{\lambda }\right\vert
^{\left\vert X_{\lambda }\right\vert }=\mathfrak{num}\left( E\right) ^{%
\mathfrak{num}\left( X\right) }.
\end{eqnarray*}

\end{proof}

\section{Ordinal numbers and numerosities\label{Ord1}}

In this section we will select a subset of the numerosities which we will
call ordinal numerosities (or simply ordinals). This set, equipped with its
natural order relation $<$, is isomorphic to the set of ordinal numbers. In
Section \ref{SPO} we will show that this correspondence is deeper than
expected since it preserves also the natural operations between ordinals.

\subsection{The ordinal numerosities}

Let $\mathbf{Num}$ be the set of numerosities.

\begin{definition}
\label{nuovi ord}The set $\mathbf{Ord}\subset \mathbf{Num}$ of ordinal
numerosities is defined as follows: $\tau \in \mathbf{Ord}$ if and only if%
\begin{equation*}
\tau =\mathfrak{num}\left( \Omega _{\tau }\right),
\end{equation*}%
where 
\begin{equation*}
\Omega _{\tau }=\left\{ x\in \mathbf{Ord\ |\ }x<\tau \right\}.
\end{equation*}
\end{definition}

It is easy to see by transfinite induction that this is a good definition.
In fact, it is immediate to check that

\begin{itemize}
\item $0\in \mathbf{Ord}$;

\item if $\tau \in \mathbf{Ord,}$ then $\tau +1=\mathfrak{num}\left( \Omega
_{\tau }\cup \left\{ \tau \right\} \right) \in \mathbf{Ord}$ (and hence $%
\mathbb{N}\subset \mathbf{Ord}$).
\end{itemize}

Moreover, if $\tau _{k}=\mathfrak{num}\left( \Omega _{k}\right) ,$ $k\in K,$
($\left\vert K\right\vert <\kappa $) are ordinal numerosities, then 
\begin{equation*}
\tau :=\mathfrak{num}\left( \dbigcup\limits_{k\in K}\Omega _{k}\right) \in 
\mathbf{Ord.}
\end{equation*}

In fact, this holds as $\dbigcup\limits_{k\in K}\Omega _{k}=\{x\in\mathbf{Ord%
}\mid x<\tau\}$: the inclusion $\dbigcup\limits_{k\in K}\Omega _{k}\subseteq
\{x\in\mathbf{Ord}\mid x<\tau\}$ holds trivially, as if $x\in\dbigcup%
\limits_{k\in K}\Omega _{k}$ then $x\in\mathbf{Ord}$ and $x\in\Omega_{k}$
for some $k$, and so $x<\tau_{k}<\tau$; conversely, if $x\in\mathbf{Ord}$ is
such that $x<\tau$, if $x\notin\dbigcup\limits_{k\in K}\Omega _{k}$ we would
have that $\Omega_{x}\supseteq \dbigcup\limits_{k\in K}\Omega _{k}$, and so
by taking numerosities we would get $x\geq\tau$, which is absurd.

\begin{definition}
If $\tau _{k},$ $k\in K,$ ($\left\vert K\right\vert <\kappa $) are ordinals,
we set 
\begin{equation*}
\underset{k\in K}{\sup }\tau _{k}=\mathfrak{num}\left( \dbigcup\limits_{k\in
K}\Omega _{\tau _{k}}\right),
\end{equation*}%
where $\tau _{k}=\mathfrak{num}\left( \Omega _{\tau _{k}}\right) .$
\end{definition}

Then $\tau =\sup_{k\in K}\tau _{k}$ is the least element in $\mathbf{Ord}$
equal or greater than every $\tau _{k},$ namely $\tau \in \mathbf{Ord}$ and 
\begin{equation}
\forall k\in K,\ \tau \geq \tau _{k};  \label{anna}
\end{equation}%
\begin{equation}
\forall k\in K,\ \forall \xi \geq \tau _{k}\Rightarrow \xi \geq \tau .
\label{sara}
\end{equation}

However $\tau $ is not the least element in $\mathbf{Num}$ greater or equal
to every $\tau _{k}\mathbf{.}$ In fact, as we have seen, if $\sup_{k\in
K}\tau _{k}$ is not a maximum, there are numerosities $\xi \in \mathbb{E}$ ,
greater that every $\tau _{k}$ and smaller than $\tau$, e.g. $%
\left(\sup_{k\in K}\tau_{k}\right)-1$.

Our construction of the ordinal numbers is similar to the construction of
Von Neumann. However, whilst a Von Neumann ordinal $\tau $ is the set of all
the Von Neumann ordinals contained in $\tau $, in our construction an
ordinal $\tau $ is the numerosity of the set of ordinals smaller than $\tau$%
. Hence, here, an ordinal number, as any other numerosity, is an atom.

Obviously, not all numerosities are ordinals: for example, $\mathfrak{num}%
\left( \mathbb{N}\right) $ is not an ordinal. In fact, if $\alpha =\mathfrak{%
num}\left( \mathbb{N}\right) $ were an ordinal then: 
\begin{eqnarray*}
\alpha &=&\mathfrak{num(}\left\{ x\in \mathbf{Ord}\ |\ x<\mathfrak{num}%
\left( \mathbb{N}\right) \right\} )=\mathfrak{num}(\mathbb{N}_{0}) \\
&=&\mathfrak{num}(\mathbb{N}\cup \left\{ 0\right\} )=\alpha +1.
\end{eqnarray*}%
In a similar way, one can prove that no infinite numerosity smaller than $%
\mathfrak{num}\left( \mathbb{N}\right) $ is an ordinal. However, $\alpha +1$
is an ordinal: 
\begin{equation*}
\alpha +1=\mathfrak{num}\left( \mathbb{N}_{0}\right) =\mathfrak{num}\left(
\{x\in \mathbf{Ord}\ |\ x<\alpha \}\right) .
\end{equation*}%
Actually $\alpha +1$ is the smallest infinite ordinal. From now on, we will
call it $\omega $.

As we expect, $\mathbf{Ord}$ is a well ordered set; in fact is $E\subset 
\mathbf{Ord,}$ the minimum is given by%
\begin{equation*}
\min \ E=\sup \left\{ x\in \mathbf{Ord}\ |\ \forall a\in E,\ x\leq a\right\}.
\end{equation*}

\subsection{Sums and products of ordinals\label{SPO}}

In this section we will show that the set of ordinal numerosities is closed
under sums and products, and we will show that there is relationship between
sums and products of ordinal numerosities and the natural operations between
Cantor ordinals.

First, we start by showing that the operations between numerosities are
consistent with the order structure over the ordinals.

\begin{theorem}
\label{figo}For all ordinal numbers $\sigma ,\tau \in \mathbf{Ord}$ we have
that 
\begin{eqnarray*}
\mathfrak{num}(\Omega _{\sigma })+\mathfrak{num}(\Omega _{\tau }) &=&%
\mathfrak{num}(\Omega _{\sigma +\tau }); \\
\mathfrak{num}(\Omega _{\sigma })\cdot \mathfrak{num}(\Omega _{\tau }) &=&%
\mathfrak{num}(\Omega _{\sigma \tau }).
\end{eqnarray*}

In particular, $\sigma +\tau \in \mathbf{Ord}$ and $\sigma \tau \in \mathbf{%
Ord}$.
\end{theorem}

\begin{proof} First let us prove that 
\begin{equation*}
\mathfrak{num}(\Omega _{\sigma +\tau })=\mathfrak{num}(\Omega _{\sigma })+%
\mathfrak{num}(\Omega _{\tau })
\end{equation*}%
acting by induction on $\tau .$ If $\tau =0,$ then this relation is obvious.
If $\tau =\gamma +1,$ then%
\begin{eqnarray*}
\mathfrak{num}(\Omega _{\sigma +\tau }) &=&\mathfrak{num}(\Omega _{\sigma
+\gamma +1})=\mathfrak{num}(\Omega _{\sigma +\gamma }\cup \left\{ \sigma
+\gamma +1\right\} )\\&=&\mathfrak{num}(\Omega _{\sigma +\gamma })+\mathfrak{num}%
(\left\{ \sigma +\gamma +1\right\} ) =\mathfrak{num}(\Omega _{\sigma })+\mathfrak{num}(\Omega _{\gamma })+1\\&=&
\mathfrak{num}(\Omega _{\sigma })+\mathfrak{num}(\Omega _{\gamma }\cup
\left\{ \gamma +1\right\} )=\mathfrak{num}(\Omega _{\sigma })+\mathfrak{num}(\Omega _{\tau }).
\end{eqnarray*}%
If $\tau =\ \sup_{k\in K}\tau _{k},\ $(where $\tau _{k}=\mathfrak{num}\left(
\Omega _{k}\right) $), is a limit ordinal, then 
\begin{equation*}
\mathfrak{num}(\Omega _{\sigma +\tau })=\sup_{k\in K}\ \mathfrak{num}(\Omega
_{\sigma +\tau _{k}})=\sup_{k\in K}\ \left[ \mathfrak{num}(\Omega _{\sigma
})+\mathfrak{num}(\Omega _{\tau _{k}})\right].
\end{equation*}%
Since $\sigma +\tau _{k}=\mathfrak{num}(\Omega _{\sigma })+\mathfrak{num}%
(\Omega _{\tau _{k}})$ is an ordinal number, $\tau $ satisfies (\ref{anna})
and (\ref{sara}) and hence, 
\begin{equation*}
\forall k\in K,\ \sigma +\tau \geq \sigma +\tau _{k};
\end{equation*}
\begin{equation*}
\forall k\in K,\ \forall \xi\in\mathbf{Ord} \ \sigma +\xi \geq \sigma +\tau _{k}\Rightarrow \sigma +\xi
\geq \sigma +\tau.
\end{equation*}%
Then, 
\begin{equation*}
\sup_{k\in K}\ \left( \sigma +\tau _{k}\right) =\sigma +\sup_{k\in K}\ \tau
_{k}
\end{equation*}%
and so 
\begin{equation*}
\mathfrak{num}(\Omega _{\sigma +\tau })=\sigma +\sup_{k\in K}\ \tau _{k}=%
\mathfrak{num}(\Omega _{\sigma })+\sup_{k\in K}\ \left[ \mathfrak{num}%
(\Omega _{\tau _{k}}\right] =\mathfrak{num}(\Omega _{\sigma })+\mathfrak{num}%
(\Omega _{\tau }).
\end{equation*}%
Similarly we act with the product. If $\tau =0,$ then this relation is
obvious. If $\tau =\gamma +1,$ then%
\begin{eqnarray*}
\mathfrak{num}(\Omega _{\sigma \tau }) &=&\mathfrak{num}(\Omega _{\sigma
\left( \gamma +1\right) })=\mathfrak{num}(\Omega _{\sigma \gamma +\sigma })=%
\mathfrak{num}(\Omega _{\sigma \gamma })+\mathfrak{num}(\Omega _{\sigma }) \\
&=&\mathfrak{num}(\Omega _{\sigma })\cdot \mathfrak{num}(\Omega _{\gamma })+%
\mathfrak{num}(\Omega _{\sigma })=\mathfrak{num}(\Omega _{\sigma })\left[ 
\mathfrak{num}(\Omega _{\gamma })+1\right] \\
&=&\mathfrak{num}(\Omega _{\sigma })\cdot \mathfrak{num}(\Omega _{\tau }).
\end{eqnarray*}%
If $\tau =\ \sup_{k\in K}\tau _{k}\ $(where $\tau _{k}=\mathfrak{num}\left(
\Omega _{k}\right) $), is a limit ordinal, then 
\begin{equation*}
\mathfrak{num}(\Omega _{\sigma \tau })=\sup_{k\in K}\ \mathfrak{num}(\Omega
_{\sigma \tau _{k}})=\sup_{k\in K}\ \left[ \mathfrak{num}(\Omega _{\sigma
})\cdot \mathfrak{num}(\Omega _{\tau _{k}})\right].
\end{equation*}%
Since $\tau $ satisfies (\ref{anna}) and (\ref{sara}), 
\begin{equation*}
\forall k\in K,\ \sigma \tau \geq \sigma \tau _{k};
\end{equation*}%
\begin{equation*}
\forall k\in K,\ \forall\xi\in\mathbf{Ord} \  \sigma \xi \geq \sigma \tau _{k}\Rightarrow \sigma \xi \geq
\sigma \tau.
\end{equation*}%
Then%
\begin{equation*}
\sup_{k\in K}\ \left( \sigma \tau _{k}\right) =\sigma \cdot \sup_{k\in K}\
\tau _{k},
\end{equation*}%
hence 
\begin{equation*}
\mathfrak{num}(\Omega _{\sigma \tau })=\sigma \cdot \sup_{k\in K}\ \tau _{k}=%
\mathfrak{num}(\Omega _{\sigma })\cdot \mathfrak{num}(\Omega _{\tau }).\qedhere
\end{equation*}

\end{proof}

\subsection{Numerosities and Cantor ordinals}

The relation with the Cantor definition of ordinal is the following: if $%
\tau \in \mathbf{Ord,}$ $\Omega _{\tau }$ is a well ordered set and hence $%
ot(\Omega _{\tau })$ (the \textit{order type} of $\Omega _{\tau })$ is a
Cantor ordinal. From now on, to avoid confusion, we will denote the Cantor
ordinals by $\bar{\tau}$ and their set by $\mathbf{COrd}$. Whilst a Cantor
ordinal is an equivalence class of well-ordered sets, in our definition an
ordinal is the numerosity of a suitable well ordered set; in particular, if
we let $\omega $ be the smallest infinite ordinal, then $\omega =\mathfrak{%
num}\left( \mathbb{N}_{0}\right) $ and $\bar{\omega}=ot(\mathbb{N}_{0})$.

Now, let us consider the map 
\begin{equation}
\Phi :\mathbf{Ord}\rightarrow \mathbf{COrd:\ \ }\Phi \left( \tau \right)
=ot(\Omega _{\tau }):=\bar{\tau}  \label{fab}
\end{equation}%
which identifies the "numerosity ordinals" with the "Cantor ordinals". So,
by construction $\Phi $ is an isomorphism between the ordered sets $\left( 
\mathbf{Ord,<}\right) $ and $\left( \mathbf{COrd,<}\right) .$

In general, the map does not preserve the operations $+,\cdot $, as $+$ and $%
\cdot $ are commutative on $\mathbf{Num}\subset \mathbb{N}^{\ast }$ but not
on $\mathbf{COrd}$. However, the situation is more interesting if we
consider the natural operations $\oplus ,\otimes $ between ordinals. We
recall that each ordinal $\bar{\sigma}$ has a unique \emph{normal form} 
\begin{equation*}
\bar{\sigma}=\sum_{n=0}^{m}\bar{\omega}^{j_{n}}a_{n}
\end{equation*}%
where $a_{n}\in \mathbb{N}$ and $n_{1}<n_{2}\Rightarrow j_{n_{1}}>j_{n_{2}}$.

By using the normal form, the \emph{natural ordinal operations} can be
defined as follows: given 
\begin{equation}
\bar{\sigma}=\sum_{n=0}^{m}\bar{\omega}^{j_{n}}a_{n}\,\,\ \mathrm{and}\,\,\ {%
\bar{\tau}}=\sum_{n=0}^{m}\bar{\omega}^{j_{n}}b_{n}  \label{abba+}
\end{equation}%
we let 
\begin{equation}
\bar{\sigma}\oplus {\bar{\tau}}=\sum_{n=0}^{m}\bar{\omega}^{j_{n}}\left(
a_{n}+b_{n}\right) \ \ \mathrm{and}\,\,\,\bar{\sigma}{\otimes \bar{\tau}}%
=\bigoplus_{n,h=0}^{m}a_{n}b_{h}\bar{\omega}^{j_{n}\oplus j_{h}},
\label{nop}
\end{equation}%
where $a_{n}+b_{n}$ and $a_{l}b_{m}$ are the usual operations on natural
numbers.

In order to compare the operations between numerosities and the natural
ordinal operations, we extend a notion used for the Cantor ordinals to the
numerosities .

\begin{definition}
An ordinal $\theta >0$ is called irreducible if 
\begin{equation*}
\sigma ,\tau ,\gamma <\theta \Rightarrow \sigma \tau +\gamma <\theta
\end{equation*}
\end{definition}

If $\theta $ is irreducible then 
\begin{equation*}
\sigma ,\tau \in \Omega _{\theta }\Rightarrow \sigma +\gamma <\theta \ \text{%
and}\ \sigma \tau <\theta ;
\end{equation*}%
we need to prove that $\sigma +\gamma \ $and$\ \sigma \tau \in \Omega
_{\theta }$.

We denote by $\theta _{j}$, $j\in \mathbf{Ord}$ the sequence of irreducible
ordinals, namely

\begin{itemize}
\item $\theta _{0}=\omega ,\ $

\item $\theta _{j}=\min $ $\left\{ x\in \mathbf{Ord}\ \mathbf{|\ }\forall
m\in \mathbb{N}_{0},\ \forall k<j,\ x>\theta _{k}^{m}\right\} .$
\end{itemize}

\begin{proposition}
\label{tsena}If $\tau \in \mathbf{Ord}$, we have that%
\begin{equation*}
\tau <\theta _{j+1}\Leftrightarrow \tau =\sum_{k=0}^{m}b_{k}\theta _{j}^{k}
\end{equation*}%
with $b_{k}\in \Omega _{\theta _{j}}.$
\end{proposition}

\begin{proof} This proof is based only on the order structure of $\mathbf{%
Ord}$ and hence it could be considered well known. However we will report it
for completeness and for the sake of the reader.

$\left( \Leftarrow \right) $ trivial.

$\left( \Rightarrow \right) $ If $\tau <\theta _{j+1},$ we take%
\begin{equation*}
n=\max \left\{ m\in \mathbb{N}_{0}\ |\ \theta _{j}^{m}\leq \tau \right\}
\end{equation*}%
Such an $m$ exists by the definition of $\theta _{j+1}.$ Then we set%
\begin{equation*}
b_{m}=\sup \left\{ x\in \Omega _{\theta _{j}}\ |\ x\theta _{j}^{m}\leq \tau
\right\}
\end{equation*}%
and%
\begin{equation*}
y_{j,m}=\tau -b_{m}\theta _{j}^{m}
\end{equation*}%
Then, 
\begin{equation}
\forall z\in \Omega _{\theta _{j}},\ y_{j,m}\leq z.  \label{xx}
\end{equation}%
Now, by inducion over $k=m-1,...,0$, we set%
\begin{equation*}
b_{k}=\sup \left\{ x\in \Omega _{\theta _{j}}\ |\
\sum_{l=k+1}^{m}b_{l}\theta _{j}^{l}+x\theta _{j}^{k}\leq \tau \right\}
\end{equation*}%
and 
\begin{equation*}
y_{j,k}=\tau -\sum_{l=k}^{m}b_{l}\theta _{j}^{l}
\end{equation*}%
so we have that,%
\begin{equation}
\forall z\in \Omega _{\theta _{j}},\ y_{j,m}\leq z  \label{xxx}
\end{equation}%
Now we claim that%
\begin{equation}
\tau -\sum_{k=0}^{m}b_{k}\theta _{j}^{k}=0  \label{cac}
\end{equation}%
In order to prove this we argue by iduction over $j\in \mathbf{Ord}\cup
\left\{ -1\right\} $ by proving that 
\begin{equation}
y_{jk}=0.  \label{x}
\end{equation}

If $j=-1,\ \tau \in \Omega _{\theta _{0}}=\mathbb{N}_{0},$ then $\forall
n\in \mathbb{N}_{0},\ y_{00}\leq 0$ and hence $y_{00}=0$. If (\ref{x}) holds 
$\forall \tau \in \Omega _{\theta _{j}},$ then by (\ref{xx}) and (\ref{xxx}%
), equality (\ref{x}) holds also for $\tau \in \Omega _{\theta _{j+1}}.$\end{proof}

\begin{corollary}
If $\sigma ,\tau \in \mathbf{Ord,}$ then $\sigma +\tau \in \mathbf{Ord}$ and 
$\sigma \tau \in \mathbf{Ord}$
\end{corollary}

\begin{proof} By Prop. \ref{tsena}, 
\begin{equation*}
\sigma =\sum_{k=0}^{n}a_{k}\theta _{j}^{k},\ \tau =\sum_{k=0}^{n}b_{k}\theta
_{j}^{k}
\end{equation*}%
for some $j\in \mathbf{Ord}$ and hence 
\begin{equation*}
\sigma +\tau =\sum_{k=0}^{n}\left( a_{k}+b_{k}\right) \theta _{j}^{k};\
\sigma \tau =\sum_{h,k=0}^{n}\left( a_{h}b_{k}\right) \theta _{j}^{h+k}.\qedhere
\end{equation*}

\end{proof}

Now we describe the sequence of the irreducible ordinal numerosities: we set

\begin{itemize}
\item $\theta _{0}=\bar{\omega}$

\item $\theta _{\bar{j}}=\sup \left\{ \theta _{\bar{k}}^{n}\ |\ n\in \mathbb{%
N},\ \bar{k}<\bar{j}\right\} $
\end{itemize}

\noindent So we have that 
\begin{eqnarray*}
\theta _{0} &=&\omega , \\
\theta _{1} &=&\omega ^{\omega }, \\
\theta _{2} &=&\omega ^{\omega ^{\omega }} \\
&&\dots \\
\theta _{j+1} &=&\theta _{j}^{\omega } \\
&&\dots \\
\theta _{\omega } &=&\varepsilon _{0} \\
&&\dots
\end{eqnarray*}%
and so on. Since the definiton of $\theta _{j}$ depends only on the order
structure of $\left( \mathbf{Ord,<}\right) $, then 
\begin{equation*}
\Phi \left( \theta _{j}\right) =\bar{\theta}_{j}.
\end{equation*}

It is well known and easy to check that any ordinal number ${\bar{\tau}}\in 
\mathbf{COrd,}\bar{\tau}<\bar{\theta}_{j+1},$ can be written as follows:%
\begin{equation*}
{\bar{\tau}}=\dbigoplus\limits_{n=0}^{m}\bar{a}_{n}\otimes \bar{\theta}%
_{j}^{n};\ \ \bar{a}_{n}<\bar{\theta}_{j}
\end{equation*}%
and the natural operations $\oplus ,\otimes $ take the following form:%
\begin{equation*}
\left( \dbigoplus\limits_{n=0}^{m}\bar{a}_{n}\otimes \bar{\theta}%
_{j}^{n}\right) \oplus \left( \dbigoplus\limits_{n=0}^{m}\bar{b}_{n}\otimes 
\bar{\theta}_{j}^{n}\right) =\dbigoplus\limits_{n=0}^{m}\left( \bar{a}%
_{n}\oplus \bar{b}_{n}\right) \otimes \bar{\theta}_{j}^{n}
\end{equation*}%
\begin{equation*}
\,\bar{\sigma}{\otimes \bar{\tau}}=\dbigoplus\limits_{n,h=0}^{m}\left( \bar{a%
}_{n}\otimes \bar{b}_{h}\right) \otimes \bar{\theta}_{j}^{n+h}
\end{equation*}

\begin{theorem}
\label{bello}The map (\ref{fab}) is an isomorphism between the semirings $(%
\mathbf{Ord,}+,\cdot )$ and $(\mathbf{COrd,}\oplus ,\otimes ),$ namely%
\begin{eqnarray*}
\Phi \left( \sigma +\tau \right) &=&\bar{\sigma}\oplus \bar{\tau} \\
\Phi \left( \sigma \tau \right) &=&\bar{\sigma}\otimes \bar{\tau}
\end{eqnarray*}
\end{theorem}

\begin{proof} Let $\tau=\sum_{k=0}^{m}b_{k}\theta_{j}^{k}$ be an ordinal numerosity. Then
\begin{eqnarray*}
\sum\limits_{n=0}^{m}\bar{b}_{n}\otimes \bar{\theta}%
_{j}^{n}&=&ot\left( \left\{ \dbigoplus\limits_{n=0}^{m}\bar{a}_{n}\otimes \bar{\theta%
}_{j}^{n}\in \mathbf{COrd}\ |\ \dbigoplus\limits_{n=0}^{m}\bar{a}_{n}\otimes 
\bar{\theta}_{j}^{n}<\dbigoplus\limits_{n=0}^{m}\bar{b}_{n}\otimes \bar{%
\theta}_{j}^{n}\right\} \right) \\
&=&ot\left( \left\{ \sum_{n=0}^{m}a_{n}\theta _{j}^{n}\in \mathbf{Ord}\ |\
\sum_{n=0}^{m}a_{n}\theta _{j}^{n}<\sum_{n=0}^{m}b_{n}\theta
_{j}^{n}\right\} \right) =ot\left(\Omega_{\tau}\right)=\bar{\tau},
\end{eqnarray*}%
namely
\begin{equation*}
\Phi \left( \tau \right) =\Phi \left( \sum_{n=0}^{m}b_{n}\theta
_{j}^{n}\right) =\dbigoplus\limits_{n=0}^{m}\bar{b}_{n}\otimes \bar{\theta}%
_{j}^{n}=\bar{\tau}.
\end{equation*}%
Hence $\Phi$ is an isomorphism. \end{proof}

\begin{remark}
Theorems \ref{bello} and \ref{figo} provide a new interpretation for the
natural operations $\oplus \ $and $\otimes $ namely%
\begin{equation*}
\bar{\sigma}\oplus {\tau }=ot(\Omega _{\sigma +\tau })\ \ \mathrm{and}\,\,\,%
\bar{\sigma}{\otimes \bar{\tau}}=ot(\Omega _{\sigma \tau })
\end{equation*}%
This fact is somewhat surprising since the operation $+$ and $\cdot $
between numerosities have been introduced in a natural way for the
numerosity theory and, \emph{a priori}, they should not have any realation
with the natural operations between ordinal numbers.
\end{remark}

Notice, however, that not all operations are the same between numerosity
ordinals and Cantor ordinals: for example, let $\bar{\varepsilon}_{0}=\bar{%
\theta}_{\bar{\omega}}$ be the Cantor ordinal that corresponds to the
numerosity ordinal $\theta _{\omega }$. If we use the ordinal
exponentiation, we have that%
\begin{equation*}
\bar{\omega}^{\bar{\varepsilon}_{0}}=\bar{\varepsilon}_{0},
\end{equation*}%
whilst on the contrary, if we use the Euclidean exponentiation, we get that%
\begin{equation*}
\omega ^{\varepsilon _{0}}>\theta _{\omega }.
\end{equation*}%
In particular the equation 
\begin{equation*}
\bar{\omega}^{x}=\bar{\varepsilon}_{0}
\end{equation*}%
in the world of Cantor ordinals has the solution $x=\bar{\varepsilon}_{0}$
while the equation%
\begin{equation*}
\omega ^{x}=\varepsilon _{0}
\end{equation*}%
in the world of Euclidean numbers, has the solution $\xi =\log _{\omega
}\varepsilon _{0}.$ $\xi $ is a well defined Euclidean number, but it is not
an ordinal number since 
\begin{equation*}
\xi <\varepsilon _{0}=\mathfrak{num}\left( \dbigcup\limits_{k<\omega }\Omega
_{\theta _{k}}\right) .
\end{equation*}%
However it is easy to prove that the ordinal exponentiation agrees with the
Euclidean exponentiation for numbers in $\Omega _{\theta _{\omega }}$.

\section{Numerosities of some denumerable sets\label{NUOVA}}

There are many different ways of definining a label set according to
Definition \ref{LB}. Different label sets might give different algebraical
properties to the numerosity; moreover, in some cases particular choices of
the label sets may lead to other concepts (e.g. Lebesgue measure for the
reals). In this and the next Sections, we want to show several examples of
these facts.

\subsection{The general strategy}

Theorem \ref{1} describes the fundamental properties of numerosities, which
are satisfied for all choices of the label set $\mathfrak{B}$ (and of the
ultrafilter $\mathcal{U}$). However, certain additional properties are
satisfied only for some choices of $\mathfrak{B}$: in fact, they depend on
the ultrafilter $\mathcal{U}$ over $\mathbf{\wp }_{fin}\left( \mathfrak{B}%
\right) $, whose existence depends on Zorn's lemma which cannot be explicit
and hence it is impossible to prove or disprove some of them. However, if we
choose a suitable label set $\mathfrak{B}$ (and, consequently we restrict
the choice of $\mathcal{U}$), it is possible to show that some properties,
as the ones mentioned in the Introduction, are satisfied independently of $%
\mathcal{U}$. The goal of Section \ref{NUOVA} is to show how a suitable
choice of $\mathfrak{B}$ allows the numerosity function to satisfy
interesting properties in many specific cases.

The smaller the set $\mathfrak{B}$ is, the more properties are satisfied by
the numerosity function. So the idea is to begin with a set $\mathfrak{B}%
_{\max }=\left\{ \mathfrak{t}\in \mathfrak{L\ }|\ \mathfrak{t}\cap \mathfrak{%
L}=\varnothing \right\} $ and to construct smaller label sets $\mathfrak{B}%
_{\max }\supset \mathfrak{B}_{1}\supset \mathfrak{B}_{2}\supset \dots $
which provide a richer and richer structure to the theory. In this paper we
are interested in the numerosity of some specific subsets of $\mathbb{N}_{0}$%
, $\mathbb{Z}$, $\mathbb{Q}$, $\mathbb{R}$ and so we will construct set of
labels $\mathfrak{B}_{\max }\supset \mathfrak{B}\left( \mathbb{N}_{0}\right)
\supset \dots \supset \mathfrak{B}\left( \mathbb{R}\right) $. Each set of
labels allows to enrich the theory with new theorems; all these theorems are
independent of the ultrafilter employed in the sense that every ultrafilter
which satisfies the finess property\footnote{%
The finess property has been introduced in the proof of the existence of the
field of Euclidean numbers.} does the job\footnote{%
Of course a smaller set of labels reduces the choice of the ultrafilter.
More precisely if $\mathfrak{B}_{1}\supset \mathfrak{B}_{2}$, an ultrafilter
constructed over $\mathfrak{B}_{2}$ makes $\mathfrak{B}_{1}$ to be a
qualified set.}.

The construction which we will present in the next sections is based on the
following definition:

\begin{definition}
\label{cu}If $\mathfrak{D}\subset \mathfrak{B}_{\max }$is a directed set
(wih respect to $\subseteq $), we define 
\begin{equation*}
\overline{\mathfrak{D}}=\mathfrak{G}\left(\{ \mathfrak{s\in B}_{\max }\mid
\exists \mathfrak{t}\in \mathfrak{D},\ \mathfrak{s}\sqsupseteq \mathfrak{t}%
\}\right),
\end{equation*}%
where $\mathfrak{G}F$ denotes the smallest lattice containing $F.$
\end{definition}

Notice that, by Definition, if $\mathfrak{D}\subset \mathfrak{B}_{\max}$ is
a directed set then

\begin{equation*}
\overline{\overline{\mathfrak{D}}}=\overline{\mathfrak{D}}.
\end{equation*}

\begin{lemma}
\label{pietro}For every $\mathfrak{D}\subset \mathfrak{B}_{\max },\ 
\overline{\mathfrak{D}}$ is a label set.
\end{lemma}

\begin{proof} Let us check that $\overline{\mathfrak{D}}$ satisfies the
properties of Definition \ref{LB}.

Property \ref{LB}.(i) holds as $\overline{\mathfrak{D}}$ is a lattice by definition.

Property \ref{LB}.(ii) holds as $\overline{\mathfrak{D}}\subseteq \mathfrak{B}_{\max}$.

Property \ref{LB}.(iii) holds as $\forall a\in \Lambda ,\exists 
\mathfrak{s}\in \mathfrak{B}_{\max },\ a\in \mathbb{V}\left( \mathfrak{s}%
\right) $ and hence, if you take any $\mathfrak{t}\in \mathfrak{D,}$ $a\in 
\mathbb{V}\left( \mathfrak{s}\cup \mathfrak{t}\right) ;$ on the other hand $%
\mathfrak{s}\cup \mathfrak{t\in }\overline{\mathfrak{D}}\mathfrak{\ }$and so 
$\bigcup_{\mathfrak{s}\in \overline{\mathfrak{D}}}\mathbb{V}\left( \mathfrak{%
s}\right) =\Lambda$.\end{proof}

The numerosity of a set depends on the set of labels $\mathfrak{B}$ and an
ultrafilter $\mathcal{U}$ consistent with $\mathfrak{B.}$ As in this section
we will discuss also coherence properties between different label sets, we
will use the notation $\mathfrak{num}_{\mathfrak{B}}^{\mathcal{U}}$ to
denote the numerosity function obtained using labels in $\mathfrak{B}$ and
the ultrafilter $\mathcal{U}$ and similarly we denote by $\ell _{\mathfrak{B}%
}(x)$ the label relative to $\mathfrak{B}$ (see Definition \ref{labling}).
This notation will be used only when there is danger of confusion, as
multiple sets of labels are used at once. We will keep to use to the simpler
notation $\mathfrak{num}$ whenever there is no danger of such confusion.

By definition, the $\overline{\mathfrak{D}}$-labelling of $a\in\Lambda$ is
given by 
\begin{equation*}
\ell _{\overline{\mathfrak{D}}}(a)=\bigcap \{\mathfrak{s}\in \overline{%
\mathfrak{D}}\mid a\in \mathbb{V}\left( \mathfrak{s}\right)\}=\bigcap \{%
\mathfrak{s}\in \mathfrak{B}_{\max}\mid a\in \mathbb{V}\left( \mathfrak{s}%
\right) \ \text{and} \ \exists \mathfrak{t}\in \mathfrak{D},\ \mathfrak{s}%
\supseteq \mathfrak{t}\};  \label{annaN}
\end{equation*}
in particular, we have that

\begin{equation}
\mathfrak{s}\in \mathfrak{D}\Rightarrow \ell _{\overline{\mathfrak{D}}}(%
\mathfrak{s})=\mathfrak{s}.  \label{annaB}
\end{equation}

\begin{proposition}
\label{coh} If $\mathfrak{D}_{1}\subset \mathfrak{D}_{2}$ then $\overline{%
\mathfrak{D}}_{1}\subset \overline{\mathfrak{D}}_{2}$, hence for all
ultrafilter $\mathcal{U}$ consistent with $\overline{\mathfrak{D}}_{1}$, for
every set $A$ in $\Lambda$ 
\begin{equation*}
\mathfrak{num}_{\overline{\mathfrak{D}}_{1}}^{\mathcal{U}}(A)=\mathfrak{num}%
_{\overline{\mathfrak{D}}_{2}}^{\mathcal{U}}(A).
\end{equation*}
\end{proposition}

\begin{proof} The inclusion $\overline{\mathfrak{D}}_{1}\subset \overline{\mathfrak{D}}_{2}$ holds trivially from Definition \ref{cu}. The consistency is immediate as if $\mathcal{U}$ contains $\overline{\mathfrak{D}}_{1}$ and $\overline{\mathfrak{D}}_{1}\subset\overline{\mathfrak{D}}_{2}$ then necessarily $\mathcal{U}$ contains $\overline{\mathfrak{D}}_{2}$. \end{proof}

\begin{lemma}
\label{B1}If $\lambda \in \overline{\mathfrak{D}},$ then $\lambda $ can be
split as follows 
\begin{equation*}
\lambda =\mathfrak{s}\cup \mathfrak{t}
\end{equation*}%
where $\mathfrak{s}\in \mathfrak{D}$ and $\mathfrak{t}$ is such that%
\begin{equation*}
\forall \sigma \in \mathfrak{D},\ \mathfrak{t}\cap \sigma =\varnothing.
\end{equation*}
\end{lemma}

\begin{proof} Given $\lambda \in \overline{\mathfrak{D}},$ we set%
\begin{equation*}
\mathfrak{s}:=\dbigcup \{\mathfrak{u}\in \mathfrak{D}\mid \mathfrak{u}%
\subset \lambda \}
\end{equation*}%
and%
\begin{equation*}
\mathfrak{t}:=\lambda \backslash \mathfrak{s}.
\end{equation*}%
Then $\mathfrak{s}\in\overline{\mathfrak{D}}$, as $\overline{\mathfrak{D}}$ is a lattice and the union defining $\mathfrak{s}$ is finite, and $\forall \sigma \in \mathfrak{D},\ \mathfrak{t}\cap
\sigma =\varnothing.$ In fact, if we set $\mathfrak{u}=\mathfrak{t}\cap \sigma$ then, as $\mathfrak{s}\supseteq 
\sigma$, we have
\begin{equation*}
\varnothing =\mathfrak{s}\cap \mathfrak{t}\supseteq \sigma%
\cap \mathfrak{t}=\mathfrak{u},
\end{equation*}
hence $\mathfrak{u}=\varnothing$.\end{proof}

\begin{remark}
\label{parra}We can look at the splitting given by Lemma \ref{B1} thinking
of $\mathfrak{B\ }$as a vector space over $\mathbb{Z}_{2};$ in this case we
can write 
\begin{equation*}
\mathfrak{B}=\mathfrak{D}\oplus \mathfrak{D}^{\perp }
\end{equation*}%
and the splitting $\lambda =\mathfrak{s}\cup \mathfrak{t\ }$implies that 
\begin{equation*}
\mathfrak{s\in D\ }\text{and\ }\mathfrak{t}\in \mathfrak{D}^{\perp }.
\end{equation*}
\end{remark}

\subsection{Numerosity of the natural numbers}

In what follows, we set $\mathbb{N}:=\{1,2,3,\dots \}$ and $\mathbb{N}_{0}:=%
\mathbb{N}\cup \{0\}$, and we let $\alpha $ denote the numerosity of $%
\mathbb{N}$. We will consider numbers in $\mathbb{N}_{0}$, as well as more
generally in $\mathbb{R}$, as atoms. Our goal is to find a label set $%
\mathfrak{B}(\mathbb{N}_{0})\subset \mathfrak{B}_{\max }$ so that we can
prove some properties of $\alpha $ and to describe the numerosity of some
subsets of $\mathbb{N}_{0}$ by functions of $\alpha $.

We define $\mathfrak{D}\left(\mathbb{N}_{0}\right)$ as follows: 
\begin{equation*}
\lambda \in \mathfrak{D}(\mathbb{N}_{0})\Leftrightarrow \exists m\in \mathbb{%
N}\ \text{such that}\ \lambda =\left\{ 0,\dots ,m!^{m!}\right\},
\end{equation*}%
and we set 
\begin{equation*}
\mathfrak{B}(\mathbb{N}_{0}):=\overline{\mathfrak{D}(\mathbb{N}_{0})}.
\end{equation*}%
which, by Lemma \ref{pietro}, is a label set. By Definition, we have that
for every $n\in \mathbb{N}_{0} $, 
\begin{equation*}
\ell (n)=\left\{ 0,\dots ,f(n)\right\} ,
\end{equation*}%
where 
\begin{equation*}
f(n):=\min \left\{ m!^{m!}\ |\ m\in \mathbb{N},\ m!^{m!}\geq n\right\} .
\end{equation*}

The main reason for such a peculiar labelling is to ensure the following
algebraical properties for $\alpha $:

\begin{proposition}
\label{basic} Let $n\in \mathbb{N}$. Then

\begin{enumerate}
\item[$(i)$] $\mathfrak{num}(\left\{ nm\ |\ m\in \mathbb{N}\right\} )=\frac{%
\alpha }{n}$;

\item[$(ii)$] $\mathfrak{num}(\left\{ m^{n}\ |\ m\in \mathbb{N}\right\}
)=\alpha ^{\frac{1}{n}}$.
\end{enumerate}
\end{proposition}

\begin{proof} (i) For $i=0,\dots ,n-1$ let 
\begin{equation*}
A_{i}=\{m\in \mathbb{N}_{0}\mid m\equiv i\mod n\}.
\end{equation*}%
Then for every $\lambda \sqsupseteq \{0,1,\dots ,n!^{n!}\}$, with $\lambda
\in \mathfrak{B}\left( \mathbb{N}_{0}\right) $, for every $0\leq i,j<n$ we
have 
\begin{equation*}
|A_{i}\cap \lambda |=|A_{j}\cap \lambda |,
\end{equation*}%
as $\lambda \cap \mathbb{N}=\{0,1,\dots ,f(m)\}$ for some $m\geq n$, and $n$
divides $f(m)$ for every such $m$. In particular, this shows that $\mathfrak{%
num}\left( A_{i}\right) =\mathfrak{num}\left( A_{j}\right) $ for every $%
0\leq i,j<n$, hence 
\begin{equation*}
\alpha =\mathfrak{num}\left( \mathbb{N}_{0}\right) =\sum_{i=0}^{n-1}%
\mathfrak{num}\left( A_{i}\right) =n\cdot \mathfrak{num}\left( A_{0}\right) .
\end{equation*}

(ii) Let $\lambda \sqsupseteq \{1,\dots ,n!^{n!}\}$, with $\lambda \in 
\mathfrak{B}\left( \mathbb{N}_{0}\right) $. As noticed in (i) above, it must be $%
\lambda \cap \mathbb{N}=\{1,\dots ,m!^{m!}\}$ for some $m\geq n$. If $a=m!^{%
\frac{m!}{n}}$, we can rewrite $\{1,\dots ,m!^{m!}\}$ as $\{1,\dots ,a^{n}\}$%
. Hence $|\left\{ m^{n}\ |\ m\in \mathbb{N}\right\} \cap \lambda |\ =a=|%
\mathbb{N}\cap \lambda |^{\frac{1}{n}}$. The thesis is reached by taking the 
$\Lambda $-limit on the above equality.\end{proof}

\begin{remark}
Of course, the choice of $\mathfrak{D}\left(\mathbb{N}_{0}\right)$ is not
intrinsic, and has been done so to make it possible to have the properties
listed in Proposition \ref{basic}. Some additional motivations for this
choice of $\mathfrak{D}\left(\mathbb{N}_{0}\right)$ can be found in \cite%
{BDN2018}; different motivations have lead the authors of \cite{BF} to make
the following different choice:%
\begin{equation*}
\lambda \in \mathfrak{D}_{1}(\mathbb{N}_{0})\Leftrightarrow \exists m\in 
\mathbb{N}\ \text{such that}\ \lambda =\left\{ 0,\dots ,2^{m}-1\right\} .
\end{equation*}
This can be seen as a feature of this approach: different algebraical
properties of the numerosity can be rather easily obtained by changing the
label set.
\end{remark}

\subsection{Numerosity of the integers\label{inte}}

We proceed as in the case of the natural numbers. We define $\mathfrak{D}(%
\mathbb{Z})$ as follows: 
\begin{equation*}
\lambda \in \mathfrak{D}\left( \mathbb{Z}\right) \Leftrightarrow \exists
m\in \mathbb{N}\ \text{such that}\ \lambda =\left\{ -m!^{m!},\dots
,m!^{m!}\right\} .
\end{equation*}%
Clearly $\mathfrak{D}\left( \mathbb{Z}\right) \subset \mathfrak{B}(\mathbb{N}%
_{0})$ and hence, by Lemma \ref{pietro}, 
\begin{equation*}
\mathfrak{B}(\mathbb{Z}):=\overline{\mathfrak{B}\left(\mathbb{N}%
_{0}\right)\cap \mathfrak{D}(\mathbb{Z})}
\end{equation*}%
is a label set. Using this label basis for every $z\in \mathbb{Z}$, 
\begin{equation*}
\ell (z)\cap \mathbb{Z}=\{-n(z),\dots ,n(z)\},
\end{equation*}%
where 
\begin{equation*}
n(z):=\min \left\{ m!^{m!}\ |\ m\in \mathbb{N},\ m!^{m!}\geq |z|\right\} .
\end{equation*}

Moreover, as $\mathfrak{B}(\mathbb{Z})\subseteq \mathfrak{B}(\mathbb{N})$,
by Proposition \ref{coh} the numerosities constructed with $\mathfrak{B}(%
\mathbb{Z})$ are coherent with those constructed with $\mathfrak{B}(\mathbb{N%
})$.

With this choice of $\mathfrak{D}\left( \mathbb{Z}\right) ,$\ $\mathfrak{num}%
\left( \mathbb{Z}\right) =2\alpha +1$ and we have that 
\begin{equation}
\mathfrak{num}\left( \mathbb{Z}_{<0}\right) =\mathfrak{num}\left( \mathbb{Z}%
_{>0}\right) =\alpha ;  \label{yout}
\end{equation}%
this equality agrees with the intuition that the positive numbers are as
many as the negative numbers.

Just as an example of a possible application, let us prove the following
result for subgroups of $\mathbb{Z}$, which reminds Lagrange's Theorem for
finite groups:

\begin{theorem}
Let $S:=m\mathbb{Z}$ be a subgroup of $(\mathbb{Z},+)$. Then%
\begin{equation}
\frac{\mathfrak{num}\left( \mathbb{Z}\right) }{\mathfrak{num}\left( S\right) 
}\sim m=\mathfrak{num}\left( \mathbb{Z}_{m}\right) .  \label{zeta}
\end{equation}
\end{theorem}

\begin{proof} By definition, $S=\{mn\mid n\in \mathbb{Z}\}$. We write $%
S=S_{+}\cup S_{-}\cup \{0\}$, where 
\begin{equation*}
S_{+}=\{a\in S\mid a>0\},S_{-}=\{a\in S\mid a<0\}.
\end{equation*}%
By Proposition (\ref{basic}) we know that $\mathfrak{num}\left( S_{+}\right)
=\frac{\alpha }{m}$, and it is trivial to show that $\mathfrak{num}\left(
S_{-}\right) =\mathfrak{num}\left( S_{+}\right) $. Hence $\mathfrak{num}%
\left( S\right) =\mathfrak{num}\left( S_{+}\right) +\mathfrak{num}\left(
S_{-}\right) +1=2\frac{\alpha }{m}+1$. As $\mathfrak{num}\left( \mathbb{Z}%
\right) =2\alpha +1,$ we have 
\begin{equation*}
\frac{\mathfrak{num}\left( \mathbb{Z}\right) }{\mathfrak{num}\left( S\right) 
}=\frac{2\alpha +1}{\frac{2\alpha }{m}+1}=\frac{2\alpha +1}{\frac{1}{m}%
\left( 2\alpha +m\right) }\sim m,
\end{equation*}%
as $\alpha $ is infinite.\end{proof}

\begin{remark}
Let us notice that, with our labelling, in the above Proposition we do not
have the equality 
\begin{equation}
\frac{\mathfrak{num}\left( \mathbb{Z}\right) }{\mathfrak{num}(S)}=m
\label{ciccia}
\end{equation}
because not all lateral classes $\left[ k\right] $ in the quotient have the
same numerosity: 
\begin{equation*}
\mathfrak{num}\left( \left[ k\right] \right) =\frac{2\alpha }{m}\ \ \text{if}%
\ \ k\neq 0;\ \ \mathfrak{num}\left( \left[ 0\right] \right) =\frac{2\alpha 
}{m}+1.
\end{equation*}%
If we want the equality in (\ref{zeta}), then we can replace $\mathfrak{D}%
\left( \mathbb{Z}\right) $ with $\mathfrak{D}_{1}\left( \mathbb{Z}\right) $
defined as follows:%
\begin{equation*}
\lambda \in \mathfrak{D}_{1}\left( \mathbb{Z}\right) \Leftrightarrow \exists
m\in \mathbb{N}\ \text{such that}\ \lambda =\left\{ -m!^{m!}+1,\dots
,m!^{m!}\right\} .
\end{equation*}%
In this case, we get (\ref{ciccia}), but $\mathfrak{num}\left( \mathbb{Z}%
\right) =2\alpha $ and the equality (\ref{yout}) is violated.
\end{remark}

\subsection{Numerosity of the rationals\label{raz}}

The labelling of $\mathbb{Z}$ given in Section \ref{inte} can be extended in
several ways to the rationals. A natural one is obtained by setting

\begin{equation*}
\mathfrak{D}\left( \mathbb{Q}\right) :=\left\{ \ \mathbb{H}_{n}\mid \exists
m\in \mathbb{N},\ n=m!^{m!}\right\},
\end{equation*}
where 
\begin{equation*}
\mathbb{H}_{n}:=\left\{\frac{a}{n}\mid a\in\mathbb{Z}, \
-n^{2}<a<n^{2}\right\}.
\end{equation*}
By Lemma \ref{pietro}, 
\begin{equation*}
\mathfrak{B}(\mathbb{Q}):=\overline{\mathfrak{D}(\mathbb{Q})}
\end{equation*}%
is a label set.

As $\mathfrak{D}\left( \mathbb{Q}\right) \subset \mathfrak{B}(\mathbb{Z})$,
by Proposition \ref{coh} the numerosities constructed with $\mathfrak{B}(%
\mathbb{Q})$ are coherent with those constructed with $\mathfrak{B}(\mathbb{Z%
})$. Using the label basis $\mathfrak{B}(\mathbb{Q})$ we have that, for
every $q\in \mathbb{Q}$, 
\begin{equation*}
\ell (q)\cap \mathbb{Q}=\mathbb{H}_{n(q)},
\end{equation*}%
where 
\begin{equation*}
n(q):=\min \left\{ m!^{m!}\ |\ m\in \mathbb{N},\ m!^{m!}\geq |q|\right\} .
\end{equation*}

This labelling has been chosen in order to have the following results:

\begin{proposition}
\label{numerosita razionali} Using the labelling $\mathfrak{B}(\mathbb{Q})$,
the following properties hold

\begin{itemize}
\item[$(i)$] for all $n\in \mathbb{N}_{0}$, $\mathfrak{num}\left( \mathbb{Q}%
\cap \left[ n,n+1\right) \right) =\alpha $;

\item[$(ii)$] for all $p,q\in \mathbb{R}$ with $p<q$, $\frac{\mathfrak{num}%
\left( \mathbb{Q}\cap \left[ p,q\right) \right) }{\mathfrak{num}\left( 
\mathbb{Q}\cap \left[ 0,1\right) \right) }\sim (p-q)$;

\item[$(iii)$] $\mathfrak{num}\left( \mathbb{Q}\right) =2\alpha ^{2}+1$.
\end{itemize}
\end{proposition}

\begin{proof} (i) Take $H_{m}\in \mathfrak{B}_{\mathbb{Q}}$ with $m$
larger than $n+1$. Then $|\left( \mathbb{Q}\cap \left[ n,n+1\right) \right)
\cap \mathbb{H}_{m}|\ =m$, hence eventually $|\left( \mathbb{Q}\cap \left[
n,n+1\right) \right) \cap \mathbb{H}_{m}|\ =|\mathbb{N}\cap \mathbb{H}_{m}|$%
, and the thesis follows by taking the $\Lambda $-limit.

(ii) Take $\mathbb{H}_{m}\in \mathfrak{B}_{\mathbb{Q}}$ with $m$ larger than 
$|p|,|q|$. Then $\left( \mathbb{Q}\cap \left[ p,q\right) \right) =(p-q)m$ if 
$p\in \mathbb{H}_{m}$, $\left( \mathbb{Q}\cap \left[ p,q\right) \right)
=(p-q)m-1$ if $p\notin \mathbb{H}_{m}$. By taking the $\Lambda $-limit we
have that either $\mathfrak{num}\left( \mathbb{Q}\cap \left[ p,q\right)
\right) =(p-q)\alpha -1$ or $\mathfrak{num}\left( \mathbb{Q}\cap \left[
p,q\right) \right) =(p-q)\alpha $, and the thesis follows as, by (i), $%
\mathfrak{num}\left( \mathbb{Q}\cap \left[ 0,1\right) \right) =\alpha $.

(iii) Let us first compute $\mathfrak{num}\left( \mathbb{Q}_{>0}\right) $.
Let $\lambda =\mathbb{H}_{n}\in \mathfrak{B}_{\mathbb{Q}}$. Then $|\mathbb{H}%
_{n}\cap \mathbb{Q}_{>0}|=n^{2}$, hence if $\varphi $ is the enumeration of $%
\mathbb{Q}_{>0}$ and $\psi $ is the enumeration of $\mathbb{N}$, we have
that $\varphi (\lambda )=\psi (\lambda )^{2}$, so%
\begin{equation*}
\mathfrak{num}\left( \mathbb{Q}_{>0}\right) =\lim_{\lambda \uparrow \Lambda
}\varphi (\lambda )=\lim_{\lambda \uparrow \Lambda }\psi ^{2}(\lambda
)=\left( \lim_{\lambda \uparrow \Lambda }\psi (\lambda )\right) ^{2}=\alpha
^{2}.
\end{equation*}%
Therefore, as each $\mathbb{H}_{n}$ is symmetrical with respect to $0$, we
also have that $\mathfrak{num}\left( \mathbb{Q}_{<0}\right) =\alpha ^{2}$,
and so

\begin{equation*}
\mathfrak{num}\left( \mathbb{Q}\right) =\mathfrak{num}\left( \mathbb{Q}%
_{<0}\right) +\mathfrak{num}\left( \mathbb{Q}_{>0}\right) +1=2\alpha ^{2}+1.\qedhere
\end{equation*}%
\end{proof}

An example of a possible application, let us prove the following result:

\begin{theorem}
Let $m_{PJ}$ denote the Peano-Jordan measure of a $m_{PJ}$-measurable set $%
E. $ Then%
\begin{equation}
m_{PJ}\left( E\right) =st\left( \frac{\mathfrak{num}\left( E\cap \mathbb{Q}%
\right) }{\mathfrak{num}\left( \left[ 0,1\right) \cap \mathbb{Q}\right) }%
\right) =st\left( \frac{1}{\mathfrak{\alpha }}\cdot \mathfrak{num}\left(
E\cap \mathbb{Q}\right) \right) .  \label{PJ}
\end{equation}
\end{theorem}

\begin{proof} If $E$ is an interval then the result follows from Proposition \ref%
{numerosita razionali}. We can extend this result to a plurinterval $%
E=\dbigcup E_{i}$ by the Sum Principle (see Theorem \ref{1}.(v)). In general, if 
$E$ is $m_{PJ}$-measurable, $\forall \varepsilon \in \mathbb{R}_{>0}$ there
are two plurintervals $A$ and $B$ such that%
\begin{equation*}
A\subseteq E\subseteq B
\end{equation*}%
\begin{equation}
\left\vert m_{PJ}(B)-m_{PJ}(E)\right\vert <\varepsilon \ \text{and}\
\left\vert m_{PJ}(E)-m_{PJ}(A)\right\vert <\varepsilon  \label{ppp}
\end{equation}%
By the Euclid Principle (see Theorem \ref{1}.(ii)), we have that%
\begin{equation*}
\mathfrak{num}\left( A\cap \mathbb{Q}\right) \subseteq \mathfrak{num}\left(
E\cap \mathbb{Q}\right) \subseteq \mathfrak{num}\left( E\cap \mathbb{Q}%
\right),
\end{equation*}%
then%
\begin{equation*}
m_{PJ}(A)\leq st\left( \frac{1}{\mathfrak{\alpha }}\cdot \mathfrak{num}%
\left( E\cap \mathbb{Q}\right) \right) \leq m_{PJ}(B).
\end{equation*}%
The conclusion follows by the inequality above, Equation (\ref{ppp}) and the arbitrariness of $\varepsilon$. \end{proof}

\section{Numerosities of non-denumerable sets\label{Card}}

\subsection{A suitable labelling\label{SL}}

Let $\mathbb{\hat{R}}^{N},$ $N\in \mathbb{N}$, $\mathbb{\hat{R}}^{N}\subset 
\mathbb{A}$, be a family of sets such that 
\begin{eqnarray*}
\mathbb{\hat{R}}^{0} &=&\mathbb{R}, \\
\mathbb{\hat{R}}^{N} &\subset &\mathbb{\hat{R}}^{N+1}
\end{eqnarray*}%
and each $\mathbb{\hat{R}}^{N+1}$ is isomorphic to $\mathbb{R}^{N}$. This
akward distinction between $\mathbb{\hat{R}}^{N}$ and $\mathbb{R}^{N}$ is
useful since, in this contest, it is easier to deal with atoms and the
points of $\mathbb{R}^{N}$ are $N$-ples. Moreover, we need to assume that
the isomorphism%
\begin{equation*}
\Psi :\mathbb{R}^{N}\rightarrow \mathbb{\hat{R}}^{N}
\end{equation*}%
preserves also the labels, namely, if $\left( x_{1},\dots,x_{N}\right) \in 
\mathbb{R}^{N},$ then 
\begin{equation}
\ell \left[ \Psi \left( x_{1},\dots,x_{N}\right) \right] =\max \left\{ \ell
\left( x_{1}\right) ,\dots,\ell \left( x_{N}\right) \right\}.  \label{psi}
\end{equation}

If $A\in \wp \left( \mathbb{\hat{R}}^{N}\right) $ for some $N$, we denote by 
$\mathcal{H}_{d}\left( A\right) $ the normalized $d$-dimensional Hausdorff
measure\footnote{%
The normalized $d$-dimensional Hausdorff measure is given by 
\begin{equation*}
\mathcal{H}_{d}\left( A\right) =N_{d}H_{d}\left( A\right),
\end{equation*}%
where $H_{d}$ is the usual Hausdorff measure and the normalization factor is
such that $\mathcal{H}_{d}\left( A\right) $ coincides with the usual
Lebesgue $m_{d}$ measure for $d\in \mathbb{N}$.} in $\mathbb{\hat{R}}^{N}$.
We introduce on $\wp \left( \mathbb{A}\right) $ the following order
relation: given $A,B\in\wp \left( \mathbb{A}\right) $, if $\left\vert
A\right\vert \neq \mathfrak{c,}$ we let $A\sqsubseteq
B\Leftrightarrow\left\vert A\right\vert \leq \left\vert B\right\vert$; if $%
\left\vert A\right\vert =\left\vert B\right\vert =\mathfrak{c,}$ we let

\begin{equation}  \label{San Gennaro aiutaci tu}
A\sqsubseteq B\Leftrightarrow \ \mathcal{H}_{d}\left( A\cap \mathbb{\hat{R}}%
^{d}\right) \leq \mathcal{H}_{d}\left( B\cap \mathbb{\hat{R}}^{d}\right).
\end{equation}%
If $A\sqsubseteq B$ and $B\sqsubseteq A,$ we will write $A\equiv B.$

We define $\mathfrak{D}\left( \mathbb{A}\right) $ as follows: $\lambda \in 
\mathfrak{D}\left( \mathbb{A}\right) $ if and only if 
\begin{equation*}
\lambda =\Xi \cup \mathfrak{A},
\end{equation*}%
where

\begin{itemize}
\item $\Xi\in\mathfrak{B}\left(\mathbb{Q}\right)$;

\item $\mathfrak{A}\in \wp _{fin}\left( \wp \left( \mathbb{A}\right)
\right); $

\item for all $A,B\in \mathfrak{A}$, the following property holds: 
\begin{equation}
A\sqsubseteq B\Rightarrow \left\vert A\cap \Xi \right\vert \leq\left\vert
B\cap \Xi \right\vert;  \label{barreto}
\end{equation}%
\begin{equation}
\left\vert A\right\vert >\aleph _{0}\Rightarrow \left\vert \Xi \cap
A\right\vert >\left\vert \mathbb{Q}\cap \Xi \right\vert ^{2}.
\label{barreto2}
\end{equation}
\end{itemize}

\begin{lemma}
\label{adamo}If $A_{1},\dots ,A_{l}\subset \mathbb{A}$ and $F\subset \mathbb{%
A}$ is a finite set, there exists $\Xi \in \wp _{fin}\left( \mathbb{A}%
\right) $ such that $F\subseteq \Xi \mathbb{\ }$and 
\begin{equation*}  \label{San Gennaro aiutaci tu}
\Xi \cup \left\{ A_{1},\dots ,A_{l}\right\} \in \mathfrak{D}\left( \mathbb{A}%
\right).
\end{equation*}
\end{lemma}

\begin{proof} Let $A_{1},\dots ,A_{l}\subset \mathbb{A}$ and $F\subset 
\mathbb{A}$ be given; first we prove the Lemma in the case in which 
\begin{equation}
A_{j}\cap A_{k}=\varnothing \ \  \text{for} \ \ j\neq k.  \label{maria}
\end{equation}%
We order the $A_{j}$'s so that 
\begin{equation*}
j<k\Rightarrow A_{j}\sqsubseteq A_{k},
\end{equation*}%
and we construct a sequence of labels 
\begin{equation*}
\lambda _{k}=\Xi _{k}\cup \left\{ A_{1},\dots ,A_{k}\right\} ,\ k\leq l
\end{equation*}%
such that $\lambda _{k}\in \mathfrak{D}\left( \mathbb{A}\right) ,\ \lambda
_{k}\subset \lambda _{k+1}$ and $\Xi _{k}\supseteq F$. We do it by
induction: for $k=1$, we set%
\begin{equation*}
\lambda _{1}=\Xi _{1}\cup \left\{ A_{1}\right\} ;\ \ \Xi _{1}=F.
\end{equation*}%
Trivially $\lambda _{1}\in \mathfrak{D}\left( \mathbb{A}\right)$, since there is
nothing to verify.

Now, if $k<l,$ in order to define $\Xi _{k+1},$ we consider four cases:

\begin{itemize}
\item[(i)] $\left\vert A_{k+1}\right\vert \leq \aleph _{0};$

\item[(ii)] $\left\vert A_{k+1}\right\vert >\aleph _{0},$ and $%
\mathcal{H}_{d}\left( A_{k+1}\cap \mathbb{\hat{R}}^{N}\right) =0;$

\item[(iii)] $\left\vert A_{k+1}\right\vert >\aleph _{0},$ and $%
\mathcal{H}_{d}\left( A_{k+1}\cap \mathbb{\hat{R}}^{N}\right) >0;$

\item[(iv)] $\left\vert A_{k+1}\right\vert >\mathfrak{c}$.
\end{itemize}

(i) We take a finite set $F_{k+1}\subset A_{k+1}$ such that $\left\vert
F_{k+1}\right\vert >\left\vert A_{k}\cap \Xi _{k}\right\vert $ and we set 
\begin{equation*}
\lambda _{k+1}=\Xi _{k+1}\cup \left\{ A_{1},\dots ,A_{k},A_{k+1}\right\} ,\
\ \Xi _{k+1}=\Xi _{k}\cup F_{k+1}.
\end{equation*}%
Then $\lambda _{k}\in \mathfrak{D}\left( \mathbb{A}\right) $ since (\ref%
{barreto}) holds.

(ii) We take a finite set $F_{k+1}\subset A_{k+1}\backslash \mathbb{Q}$ such
that \[\left\vert F_{k+1}\right\vert >\max \left\{ \left\vert A_{k}\cap \Xi
_{k}\right\vert ,\ \left\vert \mathbb{Q}\cap \Xi _{k}\right\vert
^{2}\right\}\] and we set 
\begin{equation*}
\lambda _{k+1}=\Xi _{k+1}\cup \left\{ A_{1},\dots ,A_{k+1}\right\} ,\ \ \Xi
_{k+1}=\Xi _{k}\cup F_{k+1}.
\end{equation*}%
Then (\ref{barreto}) trivially holds; moreover%
\begin{equation*}
\left\vert A_{k+1}\cap \Xi _{k+1}\right\vert =\left\vert F_{k+1}\right\vert
>\ \left\vert \mathbb{Q}\cap \Xi _{k}\right\vert ^{2}.
\end{equation*}%
Then also (\ref{barreto2}) is satisfied.

(iii) if $\mathcal{H}_{d}\left( A_{k+1}\cap \mathbb{\hat{R}}^{N}\right) >0,$
then $\left\vert A_{k+1}\right\vert \geq \mathfrak{c}$; if $\left\vert
A_{k+1}\right\vert >\mathfrak{c}$ we are in case (iv); if $\left\vert
A_{k+1}\right\vert =\mathfrak{c}$, we take $F_{k+1}\subset \left(
A_{k+1}\cap \mathbb{\hat{R}}^{N}\right) \backslash \left( \mathbb{Q}\cup \Xi
_{k}\right) $ such that $\left\vert F_{k+1}\right\vert >\max \left\{
\left\vert A_{k}\cap \Xi _{k}\right\vert ,\ \left\vert \mathbb{Q}\cap \Xi
_{k}\right\vert ^{2}\right\} $ and we argue as in point (ii).

(iv) $\left\vert A_{k+1}\right\vert >\mathfrak{c}$, we argue as in point
(ii).

Now let us consider the case in which $A_{1},\dots ,A_{l}$ does not
satisfy (\ref{maria}). In this case we take a finite family of sets $%
\left\{ B_{1},\dots ,B_{m}\right\} $ which satisfies (\ref{maria}) and such
that every $A_{k}$ is the union of some $B_{j}$'s. Then, given $\left\{
B_{1},\dots ,B_{m}\right\} $ and $\Xi ,$ we have proved that there exists 
\begin{equation*}
\bar{\lambda}=\bar{\Xi}\cup \left\{ B_{1},\dots ,B_{m}\right\} \in \mathfrak{%
D}\left( \mathbb{A}\right) ,\ \text{with} \ F\subseteq \bar{\Xi}.
\end{equation*}%
Now, it is easy to chek that 
\begin{equation*}
\lambda =\bar{\Xi}\cup \left\{ A_{1},\dots ,A_{l},B_{1},\dots ,B_{m}\right\}
\in \mathfrak{D}\left( \mathbb{A}\right)\qedhere
\end{equation*}

\end{proof}

\begin{lemma}
We have that $\mathfrak{D}\left( \mathbb{Q}\right) \subset \mathfrak{D}%
\left( \mathbb{A}\right) $ and $\left( \mathfrak{D}\left( \mathbb{A}\right)
,\subseteq \right) $ is a directed set.
\end{lemma}

\begin{proof} Given $\mathbb{H}_{n}\in \mathfrak{D}\left( \mathbb{Q}%
\right) $, eventually, there is a set $A\in \mathfrak{A}$ such that $\mathbb{%
H}_{n}\in \mathfrak{D}\left( \mathbb{A}\right) $ and so $\mathfrak{D}\left( 
\mathbb{Q}\right) \subset \mathfrak{D}\left( \mathbb{A}\right) .$

For $i=1,2$ let 
\begin{equation*}
\lambda _{i}=\Xi ^{i}\cup \left\{ A_{1}^{i},\dots ,A_{l_{i}}^{i}\right\} \in 
\mathfrak{D}_{0}\left( \mathbb{A}\right) .
\end{equation*}%
We set 
\begin{equation*}
F=\Xi ^{1}\cup \Xi ^{2};
\end{equation*}%
then, by Lemma (\ref{adamo}), we can add points to $F$ and get a set $\Xi
\supset \Xi ^{1}\cup \Xi ^{2}$ so that (\ref{barreto}) and (\ref{barreto2})
are satisfied.\end{proof}

Using Lemma \ref{pietro}, we define the label set 
\begin{equation}
\mathfrak{B}(\mathbb{A}):=\overline{\mathfrak{D}\left( \mathbb{A}\right) }.
\label{sa}
\end{equation}

\subsection{Cardinal numbers and numerosities}

A property that is natural to expect, when one has a numerosity theory for
all sets in $\Lambda $, is that it must be coherent with cardinalities,
namely it must satisfy the following property:\newline

\textbf{Cantor property:} if $A,B\subset \Lambda \backslash \mathbb{A}$ then%
\begin{equation}
\left\vert A\right\vert <\left\vert B\right\vert \Rightarrow \mathfrak{num}%
\left( A\right) <\mathfrak{num}\left( B\right) .  \label{CP}
\end{equation}

Using the labelling $\mathfrak{B}(\mathbb{A})$ defined by (\ref{sa}), the
following result holds:

\begin{theorem}
If $A,B\subset \mathbb{A}$ then%
\begin{equation*}
\left\vert A\right\vert <\left\vert B\right\vert \Rightarrow \mathfrak{num}%
\left( A\right) <\mathfrak{num}\left( B\right).
\end{equation*}
\end{theorem}

\begin{proof} Given two sets $A,B\subset \mathbb{A}$ with $\left\vert
A\right\vert <\left\vert B\right\vert $, we take a label $\lambda \supseteq
\lambda _{0}:=\Xi \cup \left\{ A,B\right\} \in \mathfrak{B}(\mathbb{A}).$
Then, by (\ref{barreto}) 
\begin{equation*}
\left\vert A\cap \lambda \right\vert =\left\vert A\cap \Xi \right\vert
<\left\vert B\cap \Xi \right\vert =\left\vert B\cap \lambda \right\vert .
\end{equation*}%
The conclusion follows taking the $\Lambda $-limit
\end{proof}

By Theorem \ref{1}, it follows that the numerosity function is well defined
for every set belonging to the family 
\begin{equation*}
K=\left\{ E\in \mathbb{V}\left( F\right) \ |\ F\in \wp _{fin}\left( \mathbb{A%
}\right) \right\},
\end{equation*}%
since $\mathfrak{D}(\mathbb{A})$ provides a label to the elements of $K.$ In
particular, by the Comparison Principle (Theorem \ref{1}.(iv)) and (\ref{psi}%
), we have that for every set $E\subset \mathbb{R}^{N},$ 
\begin{equation*}
\mathfrak{num}\left( E\right) =\mathfrak{num}\left[ \Psi \left( E\right) %
\right] .
\end{equation*}

Now, we want to extend the notion of numerosity to any set $A$ in $\Lambda $
in such a way that the Cantor property (\ref{CP}) be satisfied. The simplest
way to realize this task is to consider the family of infinite sets 
\begin{equation*}
\mathbb{S}:=\Lambda \backslash \left( \mathfrak{L}\cup \mathbb{A}\right)
\end{equation*}%
and to assign a label to each of them. We can take an injective map 
\begin{equation*}
\Phi :\mathbb{S\rightarrow A}
\end{equation*}%
and set 
\begin{equation*}
\ell (A)=\ell (\Phi \left( A\right) ).
\end{equation*}%
Then, every set in $\Lambda \backslash \mathbb{A}$ has a label in $\mathfrak{%
B}(\mathbb{A})$. By the Comparison Principle (Theorem \ref{1}.(iv)), we get
our desired final result:

\begin{theorem}
If $A,B\in \Lambda \backslash \mathbb{A}$, then%
\begin{equation*}
\left\vert A\right\vert <\left\vert B\right\vert \Rightarrow \mathfrak{num}%
\left( A\right) <\mathfrak{num}\left( B\right).
\end{equation*}
\end{theorem}

\section{Numerosity and measures\label{M}}

\subsection{The general theory}

Given a numerosity theory and a set $E\in \Lambda $, we put%
\begin{equation*}
\mu _{\gamma }(E)=st\left( \frac{\mathfrak{num}\left( E\right) }{\mathfrak{%
\gamma }}\right),
\end{equation*}%
where $\gamma\in\mathbb{N}^{\ast }$. $\mu _{\gamma }$ is called numerosity
measure. As we will see, an interesting case occurs if you take $\mathfrak{%
\gamma =num}\left( \left[ 0,1\right) \right) ^{d}\ $with $d\in \mathbb{R}%
_{\geq 0}$. In this case we will say that $\mu _{\gamma }$ is the canonical $%
d$-dimensional numerosity measure.

\begin{theorem}
\label{sc}The numerosity measure $\mu _{\gamma }$ satisfies the following
properties:

\begin{enumerate}
\item[(i)] it is finitely additive: for all sets $A,B$ 
\begin{equation*}
\mu _{\gamma }\left( A\cup B\right) =\mu _{\gamma }\left( A\right) +\mu
_{\gamma }\left( B\right) -\mu _{\gamma }\left( A\cap B\right) ;
\end{equation*}

\item[(ii)] it is superadditive, namely given a denumerable partition $%
\left\{ A_{n}\right\} _{n\in \mathbb{N}}$ of a set $A\subset \mathbb{R}$,
then 
\begin{equation*}
\mu _{\gamma }\left( A\right) \geq \sum_{n=0}^{\infty }\mu _{\gamma }\left(
A_{n}\right).
\end{equation*}
\end{enumerate}
\end{theorem}

\begin{proof} (i) This is a trivial consequence of the additivity of the
numerosity.

(ii) By Theorem \ref{1}, we have that for all $N\in \mathbb{N}$, 
\begin{equation*}
\mathfrak{num}\left( A\right) \geq \mathfrak{num}\left(
\dbigcup\limits_{n=0}^{N}A_{n}\right) =\sum_{n=0}^{N}\mathfrak{num}\left(
A_{n}\right) ,
\end{equation*}%
hence 
\begin{equation*}
st\left( \frac{\mathfrak{num}\left( A\right) }{\gamma }\right) \geq st\left(
\sum_{n=0}^{N}\frac{\mathfrak{num}\left( A_{n}\right) }{\gamma }\right)
=\sum_{n=0}^{N}st\left( \frac{\mathfrak{num}\left( A_{n}\right) }{\gamma }%
\right) ;
\end{equation*}%
therefore, 
\begin{equation*}
\mu _{\gamma }\left( A\right) \geq \sum_{n=0}^{N}\mu _{\gamma }\left(
A_{n}\right) .
\end{equation*}%
The conclusion follows taking the Cauchy limit in the above inequality for $%
N\rightarrow \infty $.\end{proof}

\subsection{Numerosity of the subsets of $\mathbb{R}^{N}$}

In this section, we will show that $\mu _{\beta }$ agrees with the Lebesgue
measure, namely, if $E$ is a Lebesgue measurable set, then%
\begin{equation}
m_{L}\left( E\right) =\mu _{\beta }\left( E\right) =st\left( \frac{\mathfrak{%
num}\left( E\right) }{\beta }\right),  \label{nana}
\end{equation}%
where%
\begin{equation}
\beta :=\mathfrak{num}\left( \left[ 0,1\right) \right). \label{nana1}
\end{equation}

First, let we show that this holds for intervals:

\begin{theorem}
\label{intervalli} The numerosity measure $\mu_{\gamma}$ is translation
invariant for any $\gamma\in\mathbb{N}^{\ast }$. In particular, if $%
\gamma=\beta$ then for any $\varepsilon=\frac{a}{b}\in[0,1)$ we have that $%
\mu_{\beta}\left([0,\frac{a}{b})\right)=\frac{a}{b}$.
\end{theorem}

\begin{proof} Let $r\in\mathbb{R}, E\subseteq\mathbb{R}$. By Property \ref{barreto}, as $E\equiv r+E$ (in the sense of the ordering $\sqsubseteq$), we have that for every $\lambda\in\mathfrak{B}\left(\mathbb{A}\right), \lambda=\Xi\cup\mathfrak{A}$ necessarily $|E\cap\Xi|=|(E+r)\cap\Xi|$. By taking the $\Lambda$-limit, we get our first claim.

As for the second, we just have to observe that $[0,1)=[0,\frac{1}{b})\cup[\frac{1}{b},\frac{2}{b})\dots\cup[\frac{b-1}{b},1)$, so by finite additivity and translation invariance we get $\mu_{\beta}\left([0,\frac{1}{b})\right)=\frac{1}{b}$, and the thesis follows as, similarly, $[0,\frac{a}{b})=[0,\frac{1}{b})\cup[\frac{1}{b},\frac{2}{b})\dots\cup[\frac{a-1}{b},1)$.\end{proof}

Moreover, we have the following property:

\begin{proposition}
\label{atto di fede} The numerosity measure $\mu _{\beta }$ is subadditive
on the $\sigma$-algebra of Lebesgue measurable sets.
\end{proposition}

\begin{proof} Let $E\in\wp(\mathbb{R})$; wlog, we assume $E\in\wp\left(\mathbb{R}_{\geq 0}\right)$, as the result for a generic $E$ will then follow easily by splitting $E=E^{+}\cup E^{-}$.
Let
\begin{equation*}
E=\bigcup_{j\in\mathbb{N}} E_{j}
\end{equation*}%
be a partition of $E$, with all $E_{j}$'s Lebesgue measurable. Let $\varepsilon=\frac{a}{b}\in[0,1)$; for $N$ large enough we have
\begin{equation*}
m_{L}\left( E\right) \leq m_{L}\left( \bigcup_{j=1}^{N}E_{j}\right) +\varepsilon.
\end{equation*}%
Now $E\cap \left[ -1,-1+\varepsilon \right] =\varnothing$, so  
\begin{equation*}
m_{L}\left( E\right) \leq m_{L}\left( \bigcup_{j=1}^{N}E_{j}\cup \left[
-1,-1+\varepsilon \right] \right).
\end{equation*}

By Property \ref{San Gennaro aiutaci tu} of our labelling, where $d=1$, we have
\begin{equation*}
\mathfrak{num}\left( E\right) \leq \mathfrak{num}\left(
\bigcup_{j=1}^{N}E_{j}\cup \left[ -1,-1+\varepsilon \right] \right).
\end{equation*}%
By Theorem \ref{intervalli} \[\mathfrak{num}\left(\bigcup_{j=1}^{N}E_{j}\cup \left[ -1,-1+\varepsilon \right] \right)=\sum_{j=1}^{N}\mathfrak{num}\left( E_{j}\right)+\mathfrak{num}\left([-1,-1+\varepsilon]\right)\sim \sum_{j=1}^{N}\mathfrak{num}\left( E_{j}\right) +\varepsilon \beta,\] hence
\begin{equation*}
\mu _{\beta }\left( E\right) \leq \sum_{j=1}^{N}\mu _{\beta }\left( E_{j}\right)
+\varepsilon \leq \sum_{j=1}^{\infty }\mu _{\beta }\left( E_{j}\right)
+\varepsilon. 
\end{equation*}%
The arbitrariness of $\varepsilon$ gives the desired inequality
\begin{equation*}
\mu _{\beta }\left( E\right) \leq \sum_{j=1}^{\infty }\mu _{\beta }\left(
E_{j}\right) .\qedhere
\end{equation*}
\end{proof}

We can now prove our desired final result:

\begin{theorem}
\label{ora et riora} $\mu_{\beta}(E)=\mu_{L}(E)$ for all Lebesgue measurable
sets $E\subseteq\mathbb{R}$.
\end{theorem}

\begin{proof} By Theorems \ref{sc},\ref{intervalli} and by Proposition \ref{atto di fede} we have that $\mu_{\beta}$, restricted to Lebesgue measurable sets, has the empty set property, it is countably additive (as it is both subadditive and superadditive), it is invariant under translation and it is normalized. Hence it must coincide with the Lebesgue measure. \end{proof}

The arguments above could be generalized to prove that for any measurable
set $A\subset \mathbb{R}^{N}$ we have that 
\begin{equation*}
m_{N}(A)=st\left( \frac{\mathfrak{num}\left( A\right) }{\beta ^{N}}\right) .
\end{equation*}%
Similarly, we can define the "fractal measure" any fractal set $A\subset 
\mathbb{R}^{N}$ as follows:%
\begin{equation*}
\mathfrak{m}_{d}(A)=st\left( \frac{\mathfrak{num}\left( A\right) }{\beta ^{d}%
}\right) ,\ \ \ d\in \left[ 0,N\right] .
\end{equation*}%
We are not gonna study this fractal measure in detail here; however, it is
not difficult to chek that $\mathfrak{m}_{d}(A)$ concides with the
normalized Hausdorff measure $H_{d}$.

\subsection{Numerosity and nonstandard measures}

It is well known that the Lebesgue measure can be realized using a counting
procedure based on hyperfinite sets: this is, e.g., at the core of the
construction of Loeb measures, which is the most known and used of such
constructions. Loeb measures were introduced in mid-70's, see \cite{Loeb};
see also \cite{Ross} for an overview of Loeb methods and applications, and 
\cite{Ross2} for an overview of other applications of nonstandard analysis
in measure theory. To confront Loeb construction with our approach, here we
shortly recall Loeb construction following Goldblatt's presentantion, see 
\cite{Goldblatt}, Section 16.8.

Let $N$ be an infinite hypernatural number, and let $S=\{\frac{k}{N}\mid
-N^{2}\leq k\leq N^{2}, k$ hyperinteger$\}$. Let $\wp_{I}(S)$ the set of
internal subsets of $S$, and for every $A\in\wp_{I}(S)$ let

\begin{equation*}
m(A):=st\left( \frac{|A|}{N}\right) ,
\end{equation*}%
where $|A|$ denotes the internal cardinality of $A$. Then $m:\wp _{I}\left(
S\right) \rightarrow \lbrack 0,+\infty ]$ is a finitely additive measure on $%
\wp _{I}\left( S\right) $. The Loeb measure is obtained by means of the
usual Carath\'{e}odory extension procedure applied to $m$ (we will denote
also the Loeb measure by $m$). What Loeb proved is that the Lebesgue measure
can be seen as a restriction of $m$, in the sense that for every Lebesgue
measurable set $X$ the Lebesgue measure $m_{L}(X)$ is equal to the Loeb
measure of the so-called pre-shadow $st^{-1}(X)$ of $X$, namely

\begin{equation*}
m_{L}(X)=st\left( m\left( st^{-1}\left( X\right) \right) \right) ,
\end{equation*}%
where $st^{-1}(X)=\{\xi \in S\mid st(\xi )\in X\}$.

The similarity between our approach is that we have that, actually, $\mu
_{\beta }$ is obtained as the standard part of a quotient similar to Loeb's
one. In fact, $\mu _{\beta }\left( A\right) =st\left( \frac{|A^{\ast }\cap
\Gamma |}{|[0,1)^{\ast }\cap \Gamma |}\right) $, where:

\begin{enumerate}
\item $|\cdot|$ denotes the internal cardinality of a set;

\item $\Gamma$ is the hyperfinite set obtained by taking $%
\lim_{\lambda\uparrow\Lambda}\lambda\cap\mathbb{R}$.
\end{enumerate}

However, in our approach the use of Carath\'{e}odory extension procedure, as
well as of pre-shadows, is substituted with the choice of a particular
labelling set, which can be equivalently seen as a particular choice of the
hyperfinite set used in the quotient. A similar result in a general
nonstandard setting was first obtained by Bernstein and Wattenberg (see \cite%
{BW}; see also \cite{CUTLAND}, Section 2, for a comparison of
Bernstein-Wattenberg's result and Loeb measures), who in fact proved that
there exists hyperfinite subsets $S\subseteq \lbrack 0,1]^{\ast }$ such that
for all Lebesgue measurable $A\subseteq \lbrack 0,1]$

\begin{equation*}
m_{L}(A)=st\left( \frac{|A^{\ast }\cap S|}{|S|}\right) .
\end{equation*}

As we said before, Theorem \ref{ora et riora} provides a new proof of the
above result by taking%
\begin{equation*}
S=\lim_{\lambda \uparrow \Lambda }\ \left( \lambda \cap \mathbb{R}\right) .
\end{equation*}

Finally, the problem of the relationship between numerosities and Lebesgue
measure in general has been addressed in \cite{BDNB1, BDNB}. In these
papers, the authors introduced the notion of "elementary numerosity" (see 
\cite{BDNB1}, Definition 1.1), that we recall:

\begin{definition*}
An elementary numerosity on a set $\Omega$ is a function $n :
\wp\left(\Omega\right)\rightarrow [0, +\infty)$ defined on all subsets of $%
\Omega$, taking values in the non-negative part of an ordered field $\mathbb{%
F}\supseteq\mathbb{R}$, and such that the following two conditions are
satisfied:

\begin{enumerate}
\item $n({x}) = 1$ for every point $x\in\Omega$;

\item $n(A \cup B) = n(A) + n(B)$ whenever A and B are disjoint.
\end{enumerate}
\end{definition*}

The main connection between the "elementary numerosity" and Lebesgue measure
is given by the following result, which is one of the instances of Theorem
3.1 in \cite{BDNB}:

\begin{theorem*}
There exists an elementary numerosity $\mathfrak{n}:\wp \left( \mathbb{R}%
\right) \rightarrow \lbrack 0,+\infty )_{\mathbb{F}}$ such that $%
m_{L}(X)=st\left( \frac{\mathfrak{n}\left( X\right) }{\mathfrak{n}\left(
[0,1)\right) }\right) $ for every Lebesgue measurable set $X$.
\end{theorem*}

Once again, Theorem \ref{ora et riora} provides another proof of the above
result, as $\mathfrak{num}$, when restricted to $\wp \left( \mathbb{R}%
\right) $ is, in fact, an elementary numerosity on $\mathbb{R}$.

The interest of Theorem \ref{ora et riora} lies on the fact that it is based
on a numerosity theory which satisfies many other additional properties.

\end{document}